\documentclass[9pt]{amsart}
\usepackage{amsmath,amsfonts,amsthm}
\usepackage{color}
\usepackage{verbatim}
\usepackage{eepic}
\usepackage{epic}
\usepackage{epsfig,subfigure}
\usepackage{epstopdf}
\usepackage[]{graphics}

\newtheorem{theorem}{Theorem}[section]
\newtheorem{lemma}{Lemma}[section]

\newtheorem{remark}{Remark}[section]
\newtheorem{corollary}{Corollary}[section]
\numberwithin{equation}{section}

\begin{document}

\newcommand{\ga}{\gamma}
\newcommand{\Ga}{\Gamma}
\newcommand{\ka}{\kappa}
\newcommand{\teps}{\tilde e}
\newcommand{\fy}{\varphi}
\newcommand{\Om}{\Omega}
\newcommand{\si}{\sigma}
\newcommand{\Si}{\Sigma}
\newcommand{\de}{\delta}
\newcommand{\De}{\Delta}
\newcommand{\la}{\lambda}
\newcommand{\La}{\Lambda}
\newcommand{\ep}{\epsilon}

\newcommand{\vth}{\vartheta}
\newcommand{\vtht}{{\widetilde \vartheta}}
\newcommand{\rh}{\varrho}
\newcommand{\rlh}{{\widetilde \varrho}}

\def\al{{\alpha}}
\def\RR{{\mathbb R}}
\def\II{{(D)}}
\newcount\icount

\def\DD#1#2{\icount=#1
  \ifnum\icount<1
  \,_{ 0}\kern -.1em D^{#2}_{\kern -.1em x}
  \else
  \,_{x}\kern -.2em D^{#2}_1
  \fi
}

\def\DDRI#1#2{\icount=#1
  \ifnum\icount<1
  \,_{-\infty}^{\kern 1em R}\kern -.2em D^{#2}_{\kern -.1em x}
  \else
  \,_{x}^R \kern -.2em D^{#2}_\infty
  \fi
}

\def\DDR#1#2{\icount=#1
  \ifnum\icount<1
 _{0}^{ \kern -.1em R} \kern -.2em D^{#2}_{\kern -.1em x}
  \else
 _{x}^{ \kern -.1em R} \kern -.2em D^{#2}_{\kern -.1em 1}
  \fi
}

\def\DDCI#1#2{\icount=#1
  \ifnum\icount<1
  \,_{-\infty}^{\kern 1em C}  \kern -.2em D^{#2}_{\kern -.1em x}
  \else
  \,_{x}^C \kern -.2em  D^{#2}_\infty
  \fi
}

\def\DDC#1#2{\icount=#1
  \ifnum\icount<1
  \,_{0}^C \kern -.2em  D^{#2}_{\kern -.1em x}
  \else
  \,_{x}^C \kern -.2em D^{#2}_1
  \fi
}

\def\Hd#1{\widetilde H^{#1}(D)}

\def\Hdi#1#2{\icount=#1
  \ifnum\icount<1
  \widetilde H_{L}^{#2}\II
  \else
  \widetilde H_{R}^{#2}\II
  \fi
}

\title[Galerkin FEM for space-fractional diffusion]
{Error Analysis of Finite Element Methods for Space-Fractional Parabolic Equations}
\author {Bangti Jin \and Raytcho Lazarov \and Joseph Pasciak \and Zhi Zhou}
\address {Department of Mathematics, University of California, Riverside, University Ave. 900,
Riverside, CA 92521 (\texttt{bangti.jin@gmail.com})}
\address {Department of Mathematics, Texas A\&M University, College Station, TX 77843-3368
({\texttt{lazarov, pasciak, zzhou@math.tamu.edu}})}

\date{started May 21, 2013; today is \today}
\maketitle

\begin{abstract}
We consider an initial/boundary value problem for one-dimensional fractional-order parabolic equations
with a space fractional derivative of Riemann-Liouville type and order $\alpha\in (1,2)$. We study
a spatial semidiscrete scheme with the standard Galerkin finite element method with piecewise linear
finite elements, as well as fully discrete schemes based on the backward Euler method and Crank-Nicolson
method. Error estimates in the $L^2\II$- and $H^{\alpha/2}\II$-norm are derived for the semidiscrete
scheme, and in the $L^2\II$-norm for the fully discrete schemes. These estimates are for both smooth
and nonsmooth initial data, and are expressed directly in terms of the smoothness of the
initial data. Extensive numerical results are presented to illustrate the theoretical results.
\end{abstract}

\section{Introduction}\label{sec:intro}
We consider the following initial/boundary value problem for a space fractional-order parabolic
differential equation (FPDE) for $u(x,t)$:
\begin{equation}\label{eqn:fpde}
   \begin{aligned}
    u_t- {\DDR0 {\al}} u&= f,\quad x\in D=(0,1), \ 0<t\le T,\\
    u(0,t)&=u(1,t)=0, \quad 0<t\le T,\\
    u(x,0)&=v,\quad x\in D,
   \end{aligned}
\end{equation}
where $\al\in(1,2)$ is the order of the derivative, $f\in L^2(0,T;L^2\II)$, and $\DDR0 \alpha u$ refers to
the Riemann-Liouville fractional derivative of order $\al$, defined in \eqref{Riemann} below, and
$T>0$ is fixed. In case of $\alpha=2$, the fractional derivative $\DDR0\alpha u$ coincides with
the usual second-order derivative $u''$ \cite{KilbasSrivastavaTrujillo:2006}, and then
model \eqref{eqn:fpde} recovers the classical diffusion equation.

The classical diffusion equation is often used to describe diffusion processes. The use of a Laplace
operator in the equation rests on a Brownian motion assumption on the random motion of individual
particles. However, over last few decades, a number of studies \cite{BensonWheatcraftMeerschaert:2000,
HatanoHatano:1998,MetzlerKlafter:2000} have shown that anomalous diffusion, in which the mean square
variances grows faster (superdiffusion) or slower (subdiffusion) than that in a
Gaussian process, offers a superior fit to experimental data observed in some
processes, e.g., viscoelastic materials, soil contamination, and underground water flow. In particular,
at a microscopic level, the particle motion might be dependent, and can frequently take very large steps,
following some heavy-tailed probability distribution. The long range correlation and large jumps can cause
the underlying stochastic process to deviate significantly from Brownian motion for the classical diffusion
process. Instead, a Levy process is considered to be more appropriate. The macroscopic
counterpart is space fractional diffusion equations (SpFDEs) \eqref{eqn:fpde}, and we refer to
\cite{BensonWheatcraftMeerschaert:2000} for the derivation and relevant physical explanations. Numerous
experimental studies have shown that SpFDEs can provide accurate description of the superdiffusion
process.

Because of the extraordinary modeling capability of SpFDEs, their accurate numerical solution  has become
an important task. A number of numerical methods, prominently the finite difference method, have been
developed for the time-dependent superdiffusion process in the literature. The finite difference scheme is
usually based on a shifted Gr\"{u}nwald formula for the Riemann-Liouville fractional derivative in space.
In \cite{TadjeranMeerschaert:2007, TadjeranMeerschaertScheffler:2006}, the stability, consistency and
convergence were shown for the
finite difference scheme with the Crank-Nicolson scheme in time.
In these works, the convergence rates
are provided under the a priori assumption that the solution $u$ to \eqref{eqn:fpde} is sufficiently smooth,
which unfortunately is not justified in general, cf. Theorem \ref{thm:fullreg}.

In this work, we develop a finite element method for \eqref{eqn:fpde}. It is based on the variational formulation
of the space fractional boundary value problem, initiated in \cite{ErvinRoop:2006,ErvinRoop:2007} and recently
revisited in \cite{JinLazarovPasciak:2013a}. We establish $L^2\II$- and $\Hd{\alpha/2}$-norm error
estimates for the space semidiscrete scheme, and $L^2\II$-norm estimates for fully discrete schemes, using
analytic semigroup theory \cite{ItoKappel:2002}. Specifically,
we obtained the following results. First, in Theorem
\ref{thm:existence} we establish the
existence and uniqueness of a weak solution $u\in L^2(0,T;\Hd{\alpha/2})$ of \eqref{eqn:fpde}
(see Section \ref{sec:prelim} for
the definitions of the space $\Hd\beta$ and the operator $A$) and in Theorem \ref{thm:fullreg}
show an enhanced regularity $u\in C((0,T];\Hdi0{\al-1+\beta})$
with $\beta\in[0,1/2)$, for $v\in L^2\II$. Second, in Theorems \ref{thm:semismooth} and \ref{thm:seminonsmooth}
we show that the semidiscrete finite element solution $u_h(t)$ with
suitable discrete initial value $u_h(0)$ satisfies the a priori error bound
\begin{equation*}
  \|u_h(t)-u(t)\|_{L^2\II}+h^{\frac{\alpha}{2}-1+
      \beta}\|u_h(t)-u(t)\|_{\Hd{\frac{\alpha}{2}}}\leq Ch^{\alpha-2+2\beta}t^{l-1}\|A^lv\|_{L^2\II}, \, \, l=0,1,
\end{equation*}
with $h$ being the mesh size and any $\beta\in[0,1/2)$. Further we derived error estimates
for the fully discrete solution $U^n$, with $\tau$ being the time step size and $t_n=n\tau$,
for the backward Euler method and Crank-Nicolson method.
For the backward Euler method, in Theorems \ref{thm:fullsmooth:euler}
and \ref{thm:fullnonsmooth:euler}, we establish the following error estimates
\begin{equation*}
  \|u(t_n)-U^n\|_{L^2\II}\leq C (h^{\al-2+2\beta} + \tau)t_n^{l-1}\| A^lv \|_{L^2\II} \quad l=0,1,
\end{equation*}
and for the Crank-Nicolson method, in Theorems \ref{thm:fullsmooth:cn} and \ref{thm:fullnonsmooth:cn}, we prove
\begin{equation*}
  \|u(t_n)-U^n\|_{L^2\II}\leq C (h^{\al-2+2\beta} + \tau^2t_n^{-1})t_n^{l-1}\| A^l v \|_{L^2\II}.
\end{equation*}
These error estimates cover both smooth and nonsmooth initial data and the bounds are directly
expressed in terms of the initial data $v$. The case of nonsmooth initial data is especially
interesting in inverse problems and optimal control.

The rest of the paper is organized as follows. In Section \ref{sec:prelim}, we introduce preliminaries on
fractional derivatives and related continuous and discrete variational formulations. Then in Section
\ref{sec:weak}, we discuss the existence and uniqueness of a weak solution to \eqref{eqn:fpde} using a
Galerkin procedure, and show the regularity pickup by the semigroup theory. Further, the properties
of the discrete semigroup $E_h(t)$ are discussed. The error analysis for the semidiscrete scheme is
carried out in Section \ref{sec:semidiscrete}, and that for fully discrete schemes based on the backward
Euler method and the Crank-Nicolson method is provided in Section \ref{sec:fullydiscrete}. Numerical
results for smooth and nonsmooth initial data are presented in Section \ref{sec:numeric}. Throughout, we
use the notation $c$ and $C$, with or without a subscript, to denote a generic constant, which may change at
different occurrences, but it is always independent of the solution $u$, time $t$, mesh size $h$
and time step size $\tau$.

\section{Fractional derivatives and variational formulation}\label{sec:prelim}
In this part, we describe fundamentals of fractional calculus, the variational problem for the source
problem with a Riemann-Liouville fractional derivative, and discuss the finite element discretization.
\subsection{Fractional derivatives}
We first briefly recall the Riemann-Liouville fractional derivative. For any positive non-integer real number
$\beta$ with $n-1 < \beta < n$, $n\in \mathbb{N}$, the left-sided Riemann-Liouville fractional
derivative $\DDR0\beta u$ of order $\beta$ of the function $u\in C^n[0,1]$ is defined by \cite[pp. 70]{KilbasSrivastavaTrujillo:2006}:
\begin{equation}\label{Riemann}
  \DDR0\beta u =\frac {d^n} {d x^n} \bigg({_0\hspace{-0.3mm}I^{n-\beta}_x} u\bigg) .
\end{equation}
Here $_0\hspace{-0.3mm}I^{\gamma}_x$ for $\gamma>0$
is the left-sided Riemann-Liouville fractional integral operator of order $\gamma$ defined by
\begin{equation*}
 ({\,_0\hspace{-0.3mm}I^\gamma_x} f) (x)= \frac 1{\Gamma(\gamma)} \int_0^x (x-t)^{\gamma-1} f(t)dt,
\end{equation*}
where $\Gamma(\cdot)$ is Euler's Gamma function defined by $\Gamma(x)=\int_0^\infty t^{x-1}e^{-t}dt$.
The right-sided versions of fractional-order integral and derivative are defined analogously, i.e.,
\begin{equation*}
  ({_x\hspace{-0.3mm}I^\gamma_1} f) (x)= \frac 1{\Gamma(\gamma)}\int_x^1 (x-t)^{\gamma-1}f(t)\,dt\quad\mbox{and}\quad
  \DDR1\beta u =(-1)^n\frac {d^n} {d x^n} \bigg({_x\hspace{-0.3mm}I^{n-\beta}_1} u\bigg) .
\end{equation*}

Now we introduce some function spaces. For any $\beta\ge 0$, we denote $H^\beta\II$ to be
the Sobolev space of order $\beta$ on the unit interval $D=(0,1)$, and $\Hd \beta $ to be the set of
functions in $H^\beta\II$ whose extension by zero to $\RR$ are in $H^\beta(\RR)$. Analogously, we define
$\Hdi 0 \beta$ (respectively, $\Hdi 1 \beta$) to be the set of functions $u$ whose extension by zero
$\tilde{u}$ is in $H^\beta(-\infty,1)$ (respectively, $H^\beta(0,\infty)$). Here for $u\in \Hdi 0
\beta$, we set $\|u\|_{\Hdi 0\beta}:=\|\tilde{u}\|_{H^\beta(-\infty,1)}$ with an analogous definition
for the norm in $\Hdi 1 \beta$. The fractional derivative operator $\DDR0\beta$ is well defined for functions in
$C^n[0,1]$, and can be extended continuously from $\Hdi0\alpha$ to $L^2\II$ (\cite[Lemma 2.6]{ErvinRoop:2006},
\cite[Theorem 2.2]{JinLazarovPasciak:2013a}).

\subsection{Variational formulation and its discretization}
Now we recall the variational formulation for the source problem
\begin{equation*}
  -\DDR0\alpha u  = f,
\end{equation*}
with $u(0)=u(1)=0$, and $f\in L^2\II$. The proper variational formulation is given by
\cite{JinLazarovPasciak:2013a}: find $u\in U\equiv \Hd{\alpha/2}$ such that
\begin{equation}\label{eqn:varrl}
  A(u,\psi) = \langle f,\psi\rangle \quad \forall \psi\in U,
\end{equation}
where the sesquilinear form $A(\cdot,\cdot)$ is given by
\begin{equation*}
   A(\fy,\psi)=-\left({\DDR0{\al/2}} \fy,\ \DDR1{\al/2}\psi\right).
\end{equation*}
It is known (\cite[Lemma 3.1]{ErvinRoop:2006}, \cite[Lemma 4.2]{JinLazarovPasciak:2013a}) that
the sesquilinear form $A(\cdot,\cdot)$ is coercive on the space $U$, i.e., there is
a constant $c_0$ such that for all $\psi\in U$
\begin{equation}\label{A-coercive}
  \Re A(\psi,\psi) \ge c_0 \| \psi\|^2_{U},
\end{equation}
where $\Re$ denotes taking the real part,
and continuous on $U$, i.e., for all $\fy,\psi\in U$
\begin{equation}\label{continuous}
    |A(\fy,\psi)| \le C_0 \| \fy \|_{U}\| \psi \|_{U}.
\end{equation}
Then by Riesz representation theorem, there exists a unique bounded linear
operator $\widetilde A: \Hd {\al/2} \rightarrow H^{-\al/2}\II$ such that
\begin{equation*}
   A(\fy,\psi)= \langle \widetilde A \fy, \psi \rangle ,\quad \forall \fy,\psi \in \Hd {\al/2}.
\end{equation*}
Define $D(A)=\{ \psi \in \Hd {\al/2}: \widetilde A \psi \in L^2\II\}$
and an operator $A : D(A)\rightarrow L^2\II$ by
\begin{equation}\label{eqn:A}
A(\fy,\psi)=(A\fy,\psi),\ \fy \in D(A),\, \psi \in \Hd {\al/2}.
\end{equation}

\begin{remark}\label{rmk::singular}
The domain $D(A)$ has a complicated structure: it consists of functions of the form $I_0^\alpha
f - (I_0^\alpha f)(1)x^{\alpha-1}$, where $f\in L^2\II$ \cite{JinLazarovPasciak:2013a}. The
term $x^{\alpha-1}\in\Hdi0{\alpha-1+\beta}$, $\beta\in[0,1/2)$, appears because it is in the
kernel of the operator $\DDR0\alpha$. Hence, $D(A) \subset \Hdi0{\al-1+\beta}\cap
\Hd {\al/2}$ and it is dense in $L^2\II$.
\end{remark}

The next result shows that the linear operator $A$ is sectorial, which means that
\begin{enumerate}
\item the resolvent set $\rho(A)$ contains the sector
$\Sigma_{\theta}=\left\{ z: \theta \le |\arg z| \le \pi \right\}$ for $\theta\in(0,\pi/2)$;
\item
$ \| (\la I-A)^{-1}  \| \le M/|\lambda|$ for $\lambda \in \Sigma_{\theta}$ and some constant $M$.
\end{enumerate}

Then we have the following important lemma (cf. \cite[pp.\,94, Theorem 3.6]{ItoKappel:2002}),
for which we sketch a proof for completeness.
\begin{lemma}\label{lem:sectorial}
The linear operator $A$ defined in \eqref{eqn:A} is sectorial on $L^2\II$.
\end{lemma}
\begin{proof}
For all $\fy \in D(A)$,
we obtain by \eqref{A-coercive} and \eqref{continuous}
\begin{equation*}
\begin{split}
|(A\fy,\fy)|\le C_0\| \fy \|_{\Hd{\al/2}}^2 \le \frac{C_0}{c_0}\Re(A\fy,\fy).
\end{split}
\end{equation*}
Thus $\mathcal{N}(A)$, the numerical range of $A$, which is defined by
\begin{equation*}
    \mathcal{N}(A)=\left\{ (A\fy,\fy): \fy \in D(A)
    ~~\text{and}~~\| \fy \|_{L^2\II}=1 \right\},
\end{equation*}
is included in the sector
$\Sigma_0 = \left\{z:0\le |\arg(z)|\le \delta_0\right\}$,
with $\delta_0=\arccos\left(c_0/C_0\right)$.

Now we choose $\delta_1\in(\delta_0,\frac{\pi}{2})$ and set
$\Sigma_{\delta_1}=\left\{z:\delta_1\le |\arg(z)| \le \pi\right\}$.
Then by \cite[p.\,310, Propositon C.3.1]{Hasse_book},
the resolvent set $\rho(A)$ contains $\Sigma_{\delta_1}$
and for all $\la\in \Sigma_{\delta_1}$
\begin{equation*}
\| (\la I-A)^{-1}  \|\le \frac{1}{\text{dist}(\la,\overline{\mathcal{N}(A)})}
\le \frac{1}{\text{dist}(\la,\Sigma_0)}\le  \frac{1}{\sin(\delta_1-\delta_0)}\frac{1}{|\la|}.
\end{equation*}
That completes the proof of this lemma.
\end{proof}

The next corollary is an immediate consequence of Lemma \ref{lem:sectorial}.
\begin{corollary}\label{cor:semigroup}
The linear operator $A$ is the infinitesimal generator of an analytic semigroup
$E(t)=e^{-At}$ on $L^2\II$.
\end{corollary}
\begin{proof}
It follows directly from Lemma \ref{lem:sectorial} and standard semigroup theory, cf.,
\cite[Theorem 3.4, Proposition 3.9 and Theorem 3.19]{ItoKappel:2002}.
\end{proof}

\subsection{Finite element discretization}\label{FEM}
We introduce a finite element approximation based on an equally spaced partition of the interval
$D$. We let $h=1/m $ be the mesh size with $m>1$ being a positive integer, and consider the nodes
$x_j=jh$, $j=0,\ldots,m$. We then define $U_h$ to be the set of continuous
functions in $U$ which are linear when restricted to the subintervals $[x_i,x_{i+1}]$,
$i=0,\ldots,m-1$, i.e.,
\begin{equation*}
   U_h = \left\{\chi\in C_0(\overline{D}): \chi \mbox{ is linear over }  [x_i,x_{i+1}], \, i=0,\ldots,m \right\}.
\end{equation*}

We define the discrete operator $A_h: U_h \rightarrow U_h$ by
\begin{equation*}
(A_h \fy, \chi) = A(\fy,\chi),\quad \forall \fy,\chi \in U_h.
\end{equation*}

The lemma below is a direct corollary of properties \eqref{A-coercive} and \eqref{continuous}
of the bilinear form $A(\cdot,\cdot) $:
\begin{lemma}\label{lem:Ah}
The discrete operator $A_h$ satisfies
\begin{equation*}
\begin{split}
    \Re(A_h\psi,\psi) &\ge c_0\|\psi \|_{\Hd {\al/2}}^2,\quad \psi \in U_h,\\
    |(A_h\fy,\psi)| &\le C_0\|\fy \|_{\Hd {\al/2}}\|\psi \|_{\Hd{\al/2}},\quad \fy,\psi \in U_h.
\end{split}
\end{equation*}
\end{lemma}
\begin{remark}\label{rmk:Ahsectorial}
By Lemma \ref{lem:Ah} and repeating the argument in the proof of Lemma
\ref{lem:sectorial}, we can show that $A_h$ is a sectorial operator on
$U_h$ with the same constant as $A$.
\end{remark}

Next we recall the Ritz projection $R_h:\Hd{\alpha/2} \rightarrow U_h$
and the $L^2\II$-projection $P_h: L^2\II\rightarrow U_h$, respectively, defined by
\begin{equation}\label{eqn:Ritz}
  \begin{aligned}
    A(R_h\psi,\chi)&=A(\psi,\chi) \quad \forall \psi\in \Hd{\alpha/2}, \ \chi\in U_h,\\
    (P_h\fy,\chi) &= (\fy,\chi)\quad \forall \fy\in L^2\II,\ \chi\in U_h.
  \end{aligned}
\end{equation}

We shall also need the adjoint problem in the error analysis.
Similar to \eqref{eqn:A}, we define the adjoint operator $A^*$ as
\begin{equation*}
    A(\fy,\psi)=(\fy,A^*\psi),\quad \forall \fy \in \Hd{\alpha/2},\ \psi \in D(A^*),
\end{equation*}
where the domain $D(A^*)$ of $A^*$ satisfies $D(A^*) \subset \widetilde
H_R^{\al-1+\beta}\II\cap \Hd {\al/2}$ and it is dense in $L^2\II$.
Further, the discrete analogue $A_h^*$ of $A^*$ is defined by
\begin{equation*}
    A(\fy,\psi)=(\fy,A_h^*\psi),\quad \forall \fy, \psi \in U_h.
\end{equation*}

\section{Variational formulation of fractional-order parabolic problem}\label{sec:weak}
The variational formulation of problem \eqref{eqn:fpde} is to find $u(t)\in U$ such that
\begin{equation}\label{variational}
   (u_t,\fy) + A(u,\fy) = (f,\fy) \quad \forall \fy \in U,
\end{equation}
and $u(0)=v$. We shall establish the well-posedness of the variational formulation \eqref{variational}
using a Galerkin procedure, and an enhanced regularity estimate via analytic semigroup theory. Further,
the properties of the discrete semigroup are discussed.

\subsection{Existence and uniqueness of the weak solution}

First we state an existence and uniqueness of a weak solution, following a Galerkin
procedure \cite{Evans:2010}. To this end, we choose an orthogonal basis $\{\omega_k(x)=
\sqrt{2}\sin k\pi x\}$ in both $L^2\II$ and $H_0^1\II$ and orthonormal in $L^2\II$. In
particular, by the construction, the $L^2\II$-orthogonal projection operator $P$ into
$\mathrm{span}\{\omega_k\}$ is stable in both $L^2\II$ and $H_0^1\II$, and by interpolation,
it is also stable in $\Hd \beta$ for any $\beta\in[0,1]$. Now we fix a positive integer
$m$, and look for a solution $u_m(t)$ of the form
\begin{equation*}
u_m(t):= \sum_{k=1}^m c_k(t) \omega_k
\end{equation*}
such that for $k=1,2\ldots,m$
\begin{equation}\label{eqn:appr1}
    c_k(0)=(v,\omega_k), \quad
    (u_m',\omega_k)+A(u_m,\omega_k)=(f,\omega_k),\quad 0\leq t\leq T.
\end{equation}
The existence and uniqueness of $u_m$ follows directly from the standard theory for ordinary
differential equation systems. With the finite-dimensional approximation $u_m$ at hand,
one can deduce the following existence and uniqueness result. The proof is rather
standard, and it is given in Appendix \ref{app:existence} for completeness.

\begin{theorem}\label{thm:existence}
Let $f\in L^2(0,T;L^2\II)$ and $v\in L^2\II$. Then there exists a unique
weak solution $u\in L^2(0,T;\Hd{\alpha/2})$ of \eqref{variational}.
\end{theorem}

\def\chii{{\chi}}
Now we study the regularity of the solution $u$ using semigroup theory \cite{ItoKappel:2002}.
By Corollary \ref{cor:semigroup} and the classical semigroup theory, the solution $u$ to the
initial boundary value problem \eqref{eqn:fpde} with $f\equiv0$ can be represented as
\begin{equation*}
  u(t) = E(t) v,
\end{equation*}
where $E(t)=e^{-tA}$ is the semigroup generated by the sectorial operator $A$, cf. Corollary
\ref{cor:semigroup}.
Then we have an improved regularity by \cite[p.\,104, Corolary 1.5]{Pazy_book}.
\begin{theorem}\label{thm:fullreg}
For every $v \in L^2\II$, the homogeneous initial-boundary value problem \eqref{variational}
(with $f=0$) has a unique solution $u(x,t) \in C([0,T];L^2\II)\cap C((0,T];D(A))$.
\end{theorem}

Further, we have the following $L^2\II$ estimate.
\begin{lemma}\label{lem:smoothing1}
There is a constant C such that
\begin{equation*}
   \| A^{\ga} E(t) \psi \|_{L^2\II} \le Ct^{-\ga} \|\psi\|_{L^2\II}.
\end{equation*}
\end{lemma}
\begin{proof}
The cases $\ga=0$ and $\ga=1$ have been proved in \cite[pp. 91, Theorem 6.4 (iii)]{Thomee:2006}.
With the contour $\Gamma=\left\{z:z=\rho e^{\pm \mathrm{i}\delta_1}, \rho\ge 0 \right\}$,
the case of $\gamma\in(0,1)$ follows by
\begin{equation*}
    \begin{split}
       \| A^{\ga} E(t) \psi \|_{L^2\II}
       & = \left| \hspace{-0.5mm}\left|\frac1{2\pi \mathrm{i}} \int_{\Gamma} z^{\ga} e^{-zt} R(z;A) \psi \, dz \right|\hspace{-0.5mm}\right|_{L^2\II}\\
        & \le C\|\psi \|_{L^2\II} \int_0^{\infty} \rho^{\ga-1}e^{-\rho t}  d\rho\le C t^{-\ga}\|\psi \|_{L^2\II}.
    \end{split}
\end{equation*}
\end{proof}

\subsection{Properties of the semigroup $E_h(t)$}

Let $E_h(t)=e^{-A_ht}$ be the semigroup generated by the operator $A_h$.
Then it satisfies a discrete analogue of Lemma \ref{lem:smoothing1}.
\begin{lemma}\label{lem:Ahprop4}
There exists a constant $C>0$ 
such that for $\chi \in U_h$
\begin{equation*}
    \| A_h^{\ga}E_h(t)\chi \|_{L^2\II}  \le Ct^{-\ga} \| \chi \|_{L^2\II}.
\end{equation*}
\end{lemma}
\begin{proof}
It follows directly from Remark \ref{rmk:Ahsectorial} and Lemma \ref{lem:smoothing1}.
\end{proof}

Last we recall the Dunford-Taylor spectral representation of a rational function $r(A_h)$ of
the operator $A_h$, when $r(z)$ is bounded in a sector in the right half
plane \cite[Lemma 9.1]{Thomee:2006}.
\begin{lemma}\label{lem:semigroup}
Let $r(z)$ be a rational function that is bounded for $|\arg z|\leq \delta_1$, $|z|\geq\epsilon>0$,
and for $|z|\geq R$. Then if $\epsilon>0$ is so small that $\{z: |z|\leq\epsilon \}\subset \rho(A_h)$, we have
\begin{equation*}
  r(A_h) = r(\infty)I + \frac{1}{2\pi\mathrm{i}}\int_{\Gamma_\epsilon\cup\Gamma_\epsilon^R\cup\Gamma^R}r(z)R(z;A_h)dz,
\end{equation*}
where $R(z;A_h)=(zI-A_h)^{-1}$ is the resolvent operator, $\Gamma_\epsilon^R = \{z:
|\arg z|=\delta_1, \epsilon\leq |z|\leq R \}$, $\Gamma_\epsilon = \{z: |z|=\epsilon,\ |\arg z|\leq \delta_1\}$, and
$\Gamma^R = \{z:\ |z|=R, \delta_1\leq |\arg z| \leq \pi\}$, and with the closed
path of integration oriented in the negative sense.
\end{lemma}

\begin{remark}\label{rem:semigroup}
The representation in Lemma \ref{lem:semigroup} holds true for any function $f(z)$
which is analytic in a neighborhood of  $\{z:|\arg z|\leq \delta_1, |z|\ge \epsilon \}$,
including at $z=\infty$.
\end{remark}

\section{Error estimates for semidiscrete Galerkin FEM}\label{sec:semidiscrete}

In this section, we derive $L^2\II$- and $\Hd {\alpha/2}$-norm error estimates for the
semidiscrete Galerkin FEM: find $u_h(t)\in U_h$ such that
\begin{equation}\label{fem}
\begin{split}
 {(  u_{h,t},\fy)}+ A(u_h,\fy)&= {(f, \fy)},
\quad \forall \fy\in U_h,\ T \ge t >0.
\quad u_h(0)=v_h,
\end{split}
\end{equation}
where $v_h\in U_h$ is an approximation to the initial data $v$. We shall discuss
the case of smooth and nonsmooth initial data, i.e. $v\in D(A)$ and $v\in L^2\II$,
separately.

\subsection{Error estimate for nonsmooth initial data}

First we consider nonsmooth initial data, i.e., $v\in L^2\II$. We follow the approach
due to Fujita and Suzuki \cite{FujitaSuzuki:1991}. First, we have the following important
lemma. Here we shall use the constant $\delta_1$ and the contour $\Gamma=\left\{z:z
=\rho e^{\pm \mathrm{i}\delta_1}, \rho\ge 0 \right\}$ defined in the proof of Lemma \ref{lem:sectorial}.
\begin{lemma}\label{lem:werror}
There exists a constant $C>0$ such that for any $ \fy\in \Hd {\al/2}$ and
$z\in \Gamma$
\begin{equation*} 
    |z| \| \fy \|_{L^2\II}^2 + \| \fy \|_{\Hd {\al/2}}^2 \le C\left|z\| \fy \|_{L^2\II}^2 -A(\fy,\fy)\right|.
\end{equation*}
\end{lemma}
\begin{proof}
We use the notation $\delta_0$ and $\delta_1$
from the proof of Lemma \ref{lem:sectorial}.
Then we choose $\delta'$ such that $\delta'\in(\delta_0,\delta_1)$
and let $c'=C_0\cos \delta'$, 
cf. Fig. \ref{fig:integral_path}(a). By setting $\gamma=c_0-c'>0$, we have
\begin{equation*}
  \Re A(\fy,\fy)-\ga \|\fy\|_{\Hd {\al/2}}^2 \ge c'\|\fy\|_{\Hd {\al/2}}^2 \ge \cos\delta'|A(\fy,\fy)|.
\end{equation*}
By dividing both sides by $\|\fy\|_{L^2\II}^2$, this yields
\begin{equation*}
   |A(\fy,\fy)|/\| \fy \|_{L^2\II}^2 \in \Sigma_{\fy}= \left\{ z:|\arg\left (z-\gamma
   \|\fy\|_{\Hd {\al/2}}^2/\| \fy \|_{L^2\II}^2\right )|  \le \delta'\right\}.
\end{equation*}
Note that for $z\in \Gamma$, there holds, cf. Fig. \ref{fig:integral_path}(a)
\begin{equation*}
   \text{dist}(z,\Sigma_{\fy}) \ge |z| \sin(\delta_1-\delta')+
   \gamma\|\fy\|_{\Hd {\al/2}}^2/\| \fy \|_{L^2\II}^2 \sin\delta'.
\end{equation*}
Consequently, for $z\in \Gamma$ we get
\begin{equation}\label{eqn:proofcontrol}
\begin{split}
  \left|z\| \fy \|_{L^2\II}^2 -A(\fy,\fy)\right|&\ge \|\fy\|_{L^2\II}^2 \text{dist} (z,\Sigma_{\fy}) \\
    & \ge |z| \|\fy\|_{L^2\II}^2 \sin(\delta_1-\delta')+\gamma\|\fy\|_{\Hd {\al/2}}^2\sin\delta'\\
   & \ge \frac1C \left( |z|\| \fy \|_{L^2\II}^2 + \| \fy \|_{\Hd {\al/2}}^2\right),
\end{split}
\end{equation}
and this completes the proof.
\end{proof}

\begin{figure}[h!]
\center
  \begin{tabular}{cc}
  \includegraphics[clip=true,width=5cm]{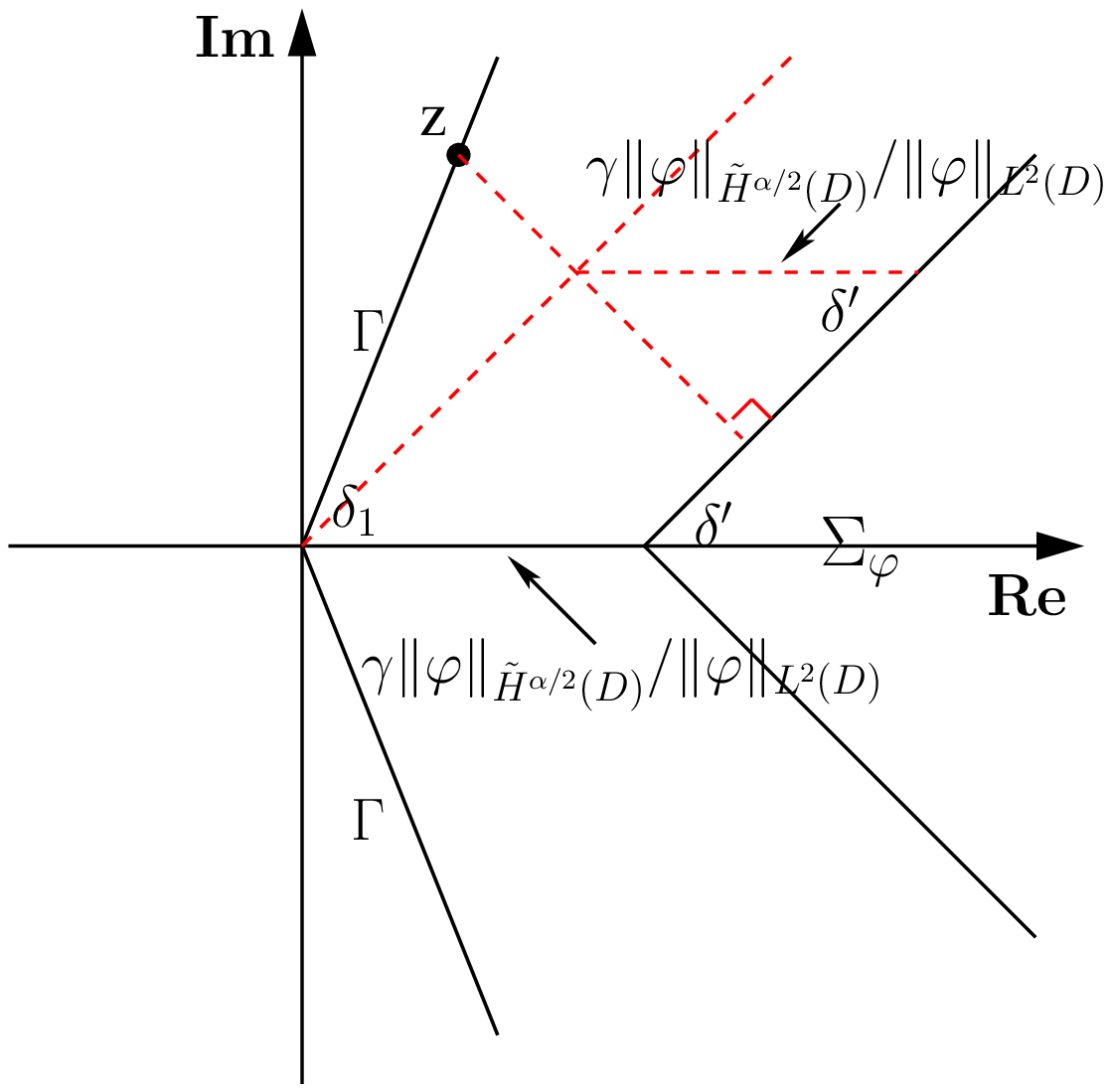}
  &\includegraphics[clip=true,width=5cm]{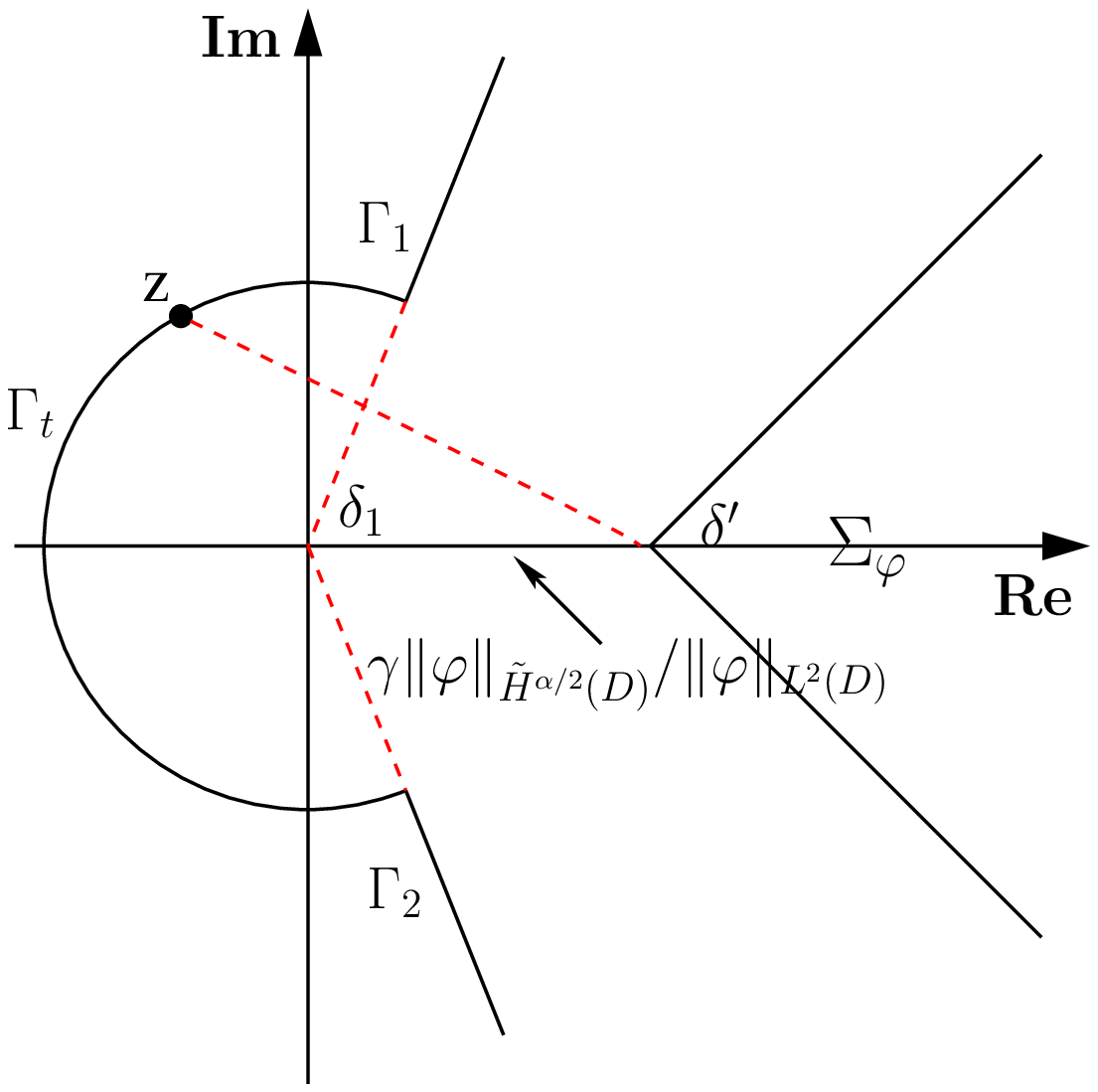}\\
  (a) & (b)\\
  \end{tabular}
  \vspace{-.2cm}
  \caption{Integration path $\Sigma_{\delta_1}$ and $\Sigma_\fy$ for (a) nonsmooth and
  (b) smooth initial data.}
  \label{fig:integral_path}
\end{figure}

The next result gives estimates on the resolvent $R(z;A)v$ and its discrete analogue.
\begin{lemma}\label{lem:wbound}
Let $ v\in L^2\II$,  $z\in \Gamma$,
 $w=R(z;A)v$, and $w_h=R(z;A_h)P_h v$. Then for $\beta\in[0,1/2)$, there holds
\begin{equation}\label{eqn:wboundHa}
    \|  w_h-w \|_{L^2\II} + h^{\al/2-1+\beta}\|  w_h-w \|_{\Hd{\al/2}}
    \le Ch^{\al-2+2\beta}\| v  \|_{L^2\II}.
\end{equation}
\end{lemma}
\begin{proof}
By the definition, $w$ and $w_h$ should respectively satisfy
\begin{equation*}
  \begin{aligned}
    z(w,\fy)-A(w,\fy)&=(v,\fy),\quad \forall \fy \in U,\\
    z(w_h,\fy)-A(w_h,\fy)&=(v,\fy),\quad \forall \fy\in U_h.
  \end{aligned}
\end{equation*}
Upon subtracting these two identities, it gives an orthogonality relation for $e=w-w_h$:
\begin{equation}\label{eqn:worthog}
    z(e,\fy) - A(e, \fy) = 0, \quad \forall \fy\in U_h.
\end{equation}
This and Lemma \ref{lem:werror} imply that for any $\chi\in U_h$
\begin{equation*}
    \begin{split}
        |z| \| e\|_{L^2\II}^2  + \| e \|_{\Hd {\al/2}}^2
        & \le C\left|z\| e \|_{L^2\II}^2 -A(e,e)\right| \\
        &   = C\left|z(e,w-\chi) -A(e,w-\chi)\right|.
     \end{split}
\end{equation*}
By taking $\chi=\pi_h w$, the finite element interpolant of $w$,
and the Cauchy-Schwarz inequality, we obtain
\begin{equation}\label{eqn:control2}
    \begin{aligned}
     |z| \| e \|_{L^2\II}^2 + \| e \|_{\Hd {\al/2}}^2
          & \le C \left(|z| h^{\al/2}\|e\|_{L^2\II}\|w\|_{\Hd {\al/2}}\right.\\
          &\qquad  +\left. h^{\al/2-1+\beta}\|e\|_{\Hd {\al/2}}\| w \|_{H^{\al-1+\beta}\II} \right).
     \end{aligned}
\end{equation}
Appealing again to Lemma \ref{lem:werror} with the choice $\fy=w$, we arrive at
\begin{equation*}
    |z| \|w \|_{L^2\II}^2 + \|w\|_{\Hd {\al/2}}^2 \le C|((zI-A)w,w)|\le C\| v \|_{L^2\II}\| w\|_{L^2\II}.
\end{equation*}
Consequently
\begin{equation}\label{eqn:wbound2}
   \begin{aligned}
    \|w \|_{L^2\II} \le C|z|^{-1}\| v \|_{L^2\II}\quad\mbox{and}\quad
    \|w \|_{\Hd {\al/2}} \le C|z|^{-1/2}\| v \|_{L^2\II}.
  \end{aligned}
\end{equation}
It remains to bound $\|w\|_{H^{\al-1+\beta}\II}$. To this end, we deduce from \eqref{eqn:wbound2} that
\begin{equation*}
\begin{split}
   \| w \|_{H^{\al-1+\beta}\II} & \le  C\| Aw \|_{L^2\II}= C\| (A-zI+zI)R(z;A)v \|_{L^2\II}\\
   &\le C\left(\| v \|_{L^2\II}+|z|\|w\|_{L^2\II}\right)\le C\| v\|_{L^2\II}.
\end{split}
\end{equation*}
It follows from this and \eqref{eqn:control2} that
\begin{equation*}
    |z| \| e\|_{L^2\II}^2 + \| e\|_{\Hd {\al/2}}^2
     \le Ch^{\al/2-1+\beta}\| v \|\left(|z|^{1/2}\| e\|_{L^2\II} + \|e\|_{\Hd {\al/2}} \right),
\end{equation*}
i.e.,
\begin{equation}\label{eqn:control3}
     |z| \|e\|_{L^2\II}^2 + \| e\|_{\Hd {\al/2}}^2\le Ch^{\al-2+2\beta}\| v \|_{L^2\II}^2.
\end{equation}
from which follows directly the $\Hd{\alpha/2}$-norm of the error $e$. Next we deduce the $L^2\II$-norm
of the error $e$ by a duality argument: given $\fy \in L^2\II$, we define $\psi$ and $\psi_h$ respectively by
\begin{equation*}
   \psi=R(z;A^*)\fy \quad\mbox{and}\quad \psi_h=R(z;A_h^*)P_h\fy.
\end{equation*}
Then by duality
\begin{equation*}
\|e \|_{L^2\II} \le \sup_{\fy \in L^2\II}\frac{|(e,\fy)|}{\|\fy\|_{L^2\II}}
=\sup_{\fy \in L^2\II}\frac{|z(e,\psi)-A(e,\psi)|}{\|\fy\|_{L^2\II}}.
\end{equation*}
Meanwhile it follows from \eqref{eqn:worthog} and \eqref{eqn:control3} that
\begin{equation*}
    \begin{split}
      |z(e,\psi)-A(e,\psi)|
        & = |z(e,\psi-\psi_h)-A(e,\psi-\psi_h)|\\
        & \le |z|\|e\|_{L^2\II}\| \psi-\psi_h \|_{L^2\II}+ C\|e\|_{\Hd{\al/2}}\| \psi-\psi_h \|_{\Hd{\al/2}}\\
        & \le Ch^{\al-2+2\beta} \| v \|_{L^2\II}\| \fy \|_{L^2\II}.
    \end{split}
\end{equation*}
This completes proof of the lemma.
\end{proof}

Now we can state our first error estimate.
\begin{theorem}\label{thm:seminonsmooth}
Let $u$ and $u_h$ be solutions of problem \eqref{variational} and \eqref{fem} with $v\in L^2\II$
and $v_h=P_h v$, respectively. Then for $t>0$, there holds for any $\beta\in[0,1/2)$:
\begin{equation*}
  \| u(t)-u_h(t) \|_{L^2\II} + h^{\al/2-1+\beta}\| u(t)-u_h(t) \|_{\Hd{\al/2}}
  \le C h^{\al-2+2\beta} t^{-1} \| v \|_{L^2\II}.
\end{equation*}
\end{theorem}
\begin{proof}
Note the error $e(t):=u(t)-u_h(t)$ can be represented as
\begin{equation*}
e(t)=\frac1{2\pi\mathrm{i}}\int_{\Gamma} e^{-zt}(w-w_h) \,dz,
\end{equation*}
where the contour $\Gamma=\left\{z:z=\rho e^{\pm \mathrm{i}\delta_1}, \rho\ge
0 \right\}$, and $w=R(z;A)v$ and $w_h=R(z;A_h)P_h v$. By Lemma \ref{lem:wbound}, we have
\begin{equation*}
\begin{split}
     \| e(t)\|_{\Hd {\al/2}} &\le C \int_{\Gamma} |e^{-zt}| \|w-w_h\|_{\Hd{\al/2}} \,dz\\
     & \le Ch^{\al/2-1+\beta}\| v \|_{L^2\II}\int_0^{\infty} e^{-\rho t \cos\delta_1}\,d\rho \le Ch^{\al/2-1+\beta}t^{-1}\| v \|_{L^2\II}.
\end{split}
\end{equation*}
A similar argument also yields the $L^2\II$-estimate.
\end{proof}

\subsection{Error estimate for smooth initial data}

Next we turn to the case of smooth initial data, i.e., $v\in D(A)$. In order to obtain a
uniform bound of the error, we employ an alternative integral representation. With $v_h
=R_h v$, then there holds
\begin{equation*}
\begin{split}
u(t)-u_h(t)&=\int_{\Gamma} e^{-zt}(R(z;A)v-R(z;A_h)R_hv )\,dz\\
&=\int_{\Gamma_{\delta_1}^t} e^{-zt}(R(z;A)v-R(z;A_h)R_hv )\,dz,
\end{split}
\end{equation*}
where $\Gamma_{\delta_1}^t=\Gamma_1\cup \Gamma_2 \cup \Gamma_t$,
$\Gamma_1=\left\{z:z=\rho e^{\mathrm{i}\delta_1}, \rho\ge t^{-1} \right\}$,
$\Gamma_2=\left\{z:z=\rho e^{-\mathrm{i}\delta_1}, \rho\ge t^{-1} \right\}$,
and $\Gamma_t=\left\{z:z=t^{-1} e^{\mathrm{i}\theta}, \delta_1 \le |\theta| \le \pi \right\}$,
cf. Fig. \ref{fig:integral_path}(b). Then using the identities
\begin{equation*}
  R(z;A)=AA^{-1}R(z;A)=A(z^{-1}R(z;A)-z^{-1}A^{-1})=z^{-1}R(z;A)A-z^{-1}I
\end{equation*}
and $\int_{\Gamma_{\delta_1}^t} e^{-st}z^{-1} \,dz=0$, the error $u(t)-u_h(t)$ can be represented as
\begin{equation}\label{eqn:smootherrorrep}
u(t)-u_h(t)=\int_{\Gamma_{\delta_1}^t} z^{-1}e^{-zt}(w-w_h) \,dz,
\end{equation}
where $w=R(z;A)Av$ and $w_h=R(z;A_h)A_hR_hv$.

\begin{lemma}\label{lem:werror2}
For any $\fy\in \Hd {\al/2}$ and $z\in \Gamma_{\delta_1}^t$, there holds
\begin{equation*} 
    |z| \| \fy \|_{L^2\II}^2 + \| \fy \|_{\Hd {\al/2}}^2 \le C\left|z\| \fy \|_{L^2\II}^2 -A(\fy,\fy)\right|.
\end{equation*}
\end{lemma}
\begin{proof}
Note that $\Gamma_1 \cup \Gamma_2 \subset \Gamma$, thus it suffices to consider $\Gamma_t$.
Set $z_t= t^{-1}e^{\mathrm{i}\delta_1}$, then it is obvious that for $z\in \Gamma_t$ and
$\fy\in \Hd {\al/2}$ we have $\text{dist}(z,\Sigma_{\fy}) \ge \text{dist}(z_t,\Sigma_{\fy})$,
cf. Fig. \ref{fig:integral_path}(b). Thus the argument in proving \eqref{eqn:proofcontrol}
yields the desired result.
\end{proof}
\begin{remark}\label{rem:wbound2}
For $ v\in L^2\II$, $ z\in \Gamma_t$, let $w=R(z;A)v$ and $w_h=R(z;A_h)P_h v$. Then
the argument in Lemma \ref{lem:wbound} and Lemma \ref{lem:werror2} yield the
estimate \eqref{eqn:wboundHa}.
\end{remark}
\begin{theorem}\label{thm:semismooth}
Let $u$ and $u_h$ be solutions of problem \eqref{variational} and \eqref{fem} with
$v\in D(A)$ and $v_h=R_h v$, respectively. Then for any $\beta\in[0,1/2)$, there holds
\begin{equation*}
   \| u(t)-u_h(t) \|_{L^2\II} + h^{\al/2-1+\beta}\|  u(t)-u_h(t) \|_{\Hd{\al/2}}
   \le C h^{\al-2+2\beta}  \| Av \|_{L^2\II}.
\end{equation*}
\end{theorem}
\begin{proof}
Let $w=R(z;A)Av$ and $w_h=R(z;A_h)A_hR_hv$. Together with the identity $A_h R_h= P_h A$,
Remark \ref{rem:wbound2} gives
\begin{equation*}
    \|  w_h-w \|_{L^2\II} + h^{\al/2-1+\beta}\|  w_h-w \|_{\Hd{\al/2}}
    \le Ch^{\al-2+2\beta}\| Av  \|_{L^2\II}.
\end{equation*}
Now it follows from this and the representation \eqref{eqn:smootherrorrep} that
\begin{equation*}
\begin{split}
     \| u(t)-u_h(t) \|_{\Hd {\al/2}} &\le C \int_{\Gamma_{\delta_1}^t}
         |z^{-1}| |e^{-zt}| \|w-w_h\|_{\Hd{\al/2}} \,dz\\
     & \le Ch^{\al/2-1+\beta}\| A v \|_{L^2\II} \int_{\Gamma_{\delta_1}^t}
         |z^{-1}| |e^{-zt}|\,dz.\\
\end{split}
\end{equation*}
It suffices to bound the integral term. First we note that
\begin{equation*}
 \int_{\Gamma_1} |z^{-1}| |e^{-zt}|\,dz
    =\int_{t^{-1}}^{\infty} \rho^{-1} e^{-\rho t \cos{\delta_1}} \, d\rho
    \le \int_{\cos{\delta_1}}^{\infty} x^{-1} e^{-x} \, dx \le C,
\end{equation*}
which is also valid for the integral on the curve $\Gamma_2$. Further, we have
\begin{equation*}
\begin{split}
 \int_{\Gamma_t} |z^{-1}| |e^{-zt}|\,dz
    & =\int_{\delta_1}^{2\pi-\delta_1} e^{\cos\theta} \, d\theta \le C.
\end{split}
\end{equation*}
Hence we obtain the $\Hd{\alpha/2}$-estimate. The $L^2\II$-estimate follows analogously.
\end{proof}

\section{Error analysis for fully discrete scheme}\label{sec:fullydiscrete}
Now we turn to error estimates for fully discrete schemes, obtained with
either the backward Euler method or the Crank-Nicolson method in time.

\subsection{Backward Euler method}\label{subsec:backward}
We first consider the backward Euler method for approximating the first-order
time derivative: for $n=1,2,\ldots,N$
\begin{equation*}
    U^n-U^{n-1}+\tau A_h U^n=0,
\end{equation*}
with $U^0=v_h$ which is an approximation of the initial data $v$. Consequently
\begin{equation}\label{eqn:backward}
    U^n=(I+\tau A_h)^{-n} v_h, \quad U^0=v_h, \quad n=1,2,...,N.
\end{equation}
By the standard energy method, the backward Euler method is unconditionally stable, i.e.,
for any $n\in\mathbb{N}$, $ \|(I+\tau A_h)^{-n}\| \le 1.$

To analyze the scheme \eqref{eqn:backward}, we need
the following smoothing property
\cite{FujitaMizutan:1976}. 
\begin{lemma}\label{lem:Ahcontrol}
For $n \in \mathbb{N}$, $n \ge \ga>0$ and $s>0$,
there exists a constant $C>0$, depending on $\ga$ only, such that
\begin{equation}\label{eqn:backeulersmoothing}
    \| A_h^{\ga}(I+sA_h)^{-n} \| \le Cn^{-\ga}s^{-\ga}.
\end{equation}
\end{lemma}
\begin{proof}
Let $r(z)=\frac1{1+z}$. Then by \cite[Lemma 9.2]{Thomee:2006}, for an arbitrary $R>0$
and $\theta \in (0,\frac{\pi}{2})$, there exist constants $c,\ C>0$ and $\ep \in (0,1)$ such that
\begin{equation}\label{eqn:rationalbound}
\begin{split}
    |r(z)| \le \left\{\begin{array}{ll}
            e^{C|z|}, & \forall |z| \le \ep,\\
             e^{-c|z|},  & \forall |z|\le R,\, |\arg z|\le \theta.
      \end{array}\right.
\end{split}
\end{equation}
Clearly, \eqref{eqn:backeulersmoothing} is equivalent to
$
  \| (nsA_h)^{\ga}r(sA_h)^n \| \le C. 
$
The fact that $A_h$ is sectorial implies that $sA_h$, $s>0$, is also sectorial on $X_h$.
Hence it suffices to show
\begin{equation*}
  \|(nA_h)^{\ga}r(A_h)^n \| \le C,
\end{equation*}
Let $F_n(z)=(nz)^{\ga}r(z)^n$.
Since $r(\infty)=0$, by Lemma \ref{lem:semigroup} and Remark \ref{rem:semigroup}
\begin{equation*}
   F_n(A_h)=\frac{1}{2\pi\mathrm{i}} \int_{\Ga_{\ep/n}\cup{\Ga_{\ep/n}^{nR}}\cup{\Ga^{nR}}}F_n(z)R(z;A_h)\,dz.
\end{equation*}
First, by \eqref{eqn:rationalbound}, we deduce that for $z \in \Ga_{\ep/n}$
\begin{equation*}
  |F_n(z)| \le (n|z|)^{\ga} e^{cn|z|} = \ep ^ {\ga} e^{c\ep} \le C .
\end{equation*}
Thus we have
\begin{equation*}
\bigg|\hspace{-0.6mm}\bigg|\frac{1}{2\pi \mathrm{i}} \int_{\Ga_{\ep/n}}F_n(z)R(z;A_h)\,dz \bigg|\hspace{-0.6mm}\bigg|
\le C \frac{\ep}{n} \sup_{z\in\Ga_{\ep/n}} \| R(z;A_h) \| \le C.
\end{equation*}
Next, we note
\begin{equation*}
\begin{split}
    \bigg|\hspace{-0.6mm}\bigg|\frac{1}{2\pi\mathrm{i}} &\int_{\Ga_{\ep/n}^{nR}}F_n(z)R(z;A_h)\,dz\bigg|\hspace{-0.6mm}\bigg|
    \le C \int_{\ep/n}^{nR} (n\rh)^\ga e^{-cn\rh} \rh^{-1}\,d\rh\\
    &\le C \int_{\ep}^{n^2R} \rho^{\ga-1} e^{-\rho}\,d\rho
    \le C \int_{0}^{\infty} \rho^{\ga-1} e^{-\rho}\,d\rho \leq C.
\end{split}
\end{equation*}
Last, there holds $|1+nz|^{-1} \le C (n|z|)^{-1}$ for $|z|\ge1$.
Hence for $z \in \Ga^{nR}$,
\begin{equation*}
   |F_n(z)| \le C n^{2\ga-n}R^{\ga-n} \le C,\quad \forall n\ge \ga.
\end{equation*}
Thus we have the following bound for the integral on the curve $\Gamma^{nR}$:
\begin{equation*}
    \bigg|\hspace{-0.6mm}\bigg|\frac{1}{2\pi\mathrm{i}} \int_{\Ga^{nR}}F_n(z)R(z;A_h)\,dz  \bigg|\hspace{-0.6mm}\bigg|
             \le C nR \sup_{z\in\Ga^{nR}} \| R(z;A_h) \| \le C.
\end{equation*}
This completes the proof of the lemma.
\end{proof}

Now we derive an error estimate for the fully discrete scheme
\eqref{eqn:backward} in case of smooth initial data, i.e., $v\in D(A)$.
\begin{theorem}\label{thm:fullsmooth:euler}
Let $u$ and $U^n$ be solutions of problem \eqref{variational} and \eqref{eqn:backward} with $v\in D(A)$
and $U^0=R_h v$, respectively. Then for $t_n=n\tau$ and any $\beta\in[0,1/2)$, there holds
\begin{equation*}
\| u(t_n)- U^n \|_{L^2\II} \le C (h^{\al-2+2\beta} + \tau)\| Av \|_{L^2\II}.
\end{equation*}
\end{theorem}
\begin{proof}
Note that the error $e^n=u(t_n)-U^n$ can be split into
\begin{equation*}
   e^n=  (u(t_n) - u_h(t_n)) + (u_h(t_n) - U^n) := \rlh^n + \vtht^n,
\end{equation*}
where $u_h$ denotes the semidiscrete Galerkin solution with $v_h=R_h v$.
By Theorem \ref{thm:semismooth}, the term $\rlh^n$ satisfies the following estimate
\begin{equation*}
   \|\rlh^n \|_{L^2\II} \le C h^{\al-2+2\beta}\| Av\|_{L^2\II}.
\end{equation*}
Next we bound the term $\vtht^n$. Note that for $n\ge1$,
\begin{equation}\label{eqn:euler-err}
\begin{split}
\vtht^n & = E_h(n\tau)-(I+\tau A_h)^{-n} v_h \\
 & = -\int_0^{\tau} \frac{d}{ds}\left(E_h(n(\tau-s))(I+sA_h)^{-n}v_h\right)\, ds\\
 & = -\int_0^{\tau}nsA_h^2 E_h(n(\tau-s))(I+sA_h)^{-n-1}v_h\, ds.
\end{split}
\end{equation}
Then by Lemmas \ref{lem:Ahprop4} and \ref{lem:Ahcontrol} we have
\begin{equation*}
\begin{split}
  \|\vtht^n\|_{L^2\II} & \le C n^{1/2} \int_0^{\tau} s(\tau-s)^{-1/2}\|A_h^{3/2} (I+sA_h)^{-n-1} R_h v \|_{L^2\II} \, ds \\
   & \le C n^{1/2} \int_0^{\tau} s^{1/2} (n+1)^{-1/2} (\tau-s)^{-1/2} \|A_h R_h v \|_{L^2\II} \, ds \\
   & \le C\tau\| A_h R_h v \|_{L^2\II}.
\end{split}
\end{equation*}
The desired result follows from the identity $A_hR_h=P_hA$ and the $L^2\II$-stability of the projection $P_h$.
\end{proof}

Next we give an error estimate for $L^2\II$ initial data $v$.
\begin{theorem}\label{thm:fullnonsmooth:euler}
Let $u$ and $U^n$ be solutions of problem \eqref{variational} and \eqref{eqn:backward} with $v\in L^2\II$
and $U^0=P_h v$, respectively. Then for $t_n=n\tau$ and any $\beta\in[0,1/2)$, there holds
\begin{equation*}
  \| u(t_n)- U^n \|_{L^2\II} \le C (h^{\al-2+2\beta} + \tau)t_n^{-1}\| v \|_{L^2\II}.
\end{equation*}
\end{theorem}
\begin{proof}
Like before, we split the error $e^n=u(t_n)-U^n$ into
\begin{equation}\label{eqn:fullsplit}
   e^n=  (u(t_n) - u_h(t_n)) + (u_h(t_n) - U^n) := \rlh^n + \vtht^n,
\end{equation}
where $u_h$ denotes the semidiscrete Galerkin solution with $v_h=P_h v$.
In view of Theorem \ref{thm:seminonsmooth}, 
it remains to estimate the term $\vtht^n$. By identity \eqref{eqn:euler-err} and Lemmas \ref{lem:Ahcontrol} and
\ref{lem:Ahprop4}, we have for $n\ge1$
\begin{equation*}
\begin{split}
\|\vtht^n\|_{L^2\II} & \le C n \int_0^{\tau} s \|A_h^{3/2} (I+sA_h)^{-n-1} A_h^{1/2} E_h(n(\tau-s)) P_h v \|_{L^2\II} \, ds \\
& \le C n \int_0^{\tau} s s^{-3/2} (n+1)^{-3/2} \|A_h^{1/2} E_h(n(\tau-s)) P_h v \|_{L^2\II} \, ds \\
& \le C n^{-1/2} \int_0^{\tau} s^{-1/2} n^{-1/2}(\tau-s)^{-1/2} \| P_h v \|_{L^2\II} \, ds  \le C\tau t_n^{-1}\| v \|_{L^2\II}.\\
\end{split}
\end{equation*}
This completes the proof of the theorem.
\end{proof}

\subsection{Crank-Nicolson method}\label{ssec:fullcn}
Now we turn to the fully discrete scheme based on the Crank-Nicolson method. It reads
\begin{equation*}
    U^n-U^{n-1}+\tau A_h U^{n-1/2}=0, \quad U^0=v_h, \quad n=1,2,...,N,
\end{equation*}
where $U^{n-1/2}=\frac12( U^{n} + U^{n-1 })$. Therefor we have
\begin{equation}\label{eqn:Crank-Nicolson}
    U^n=\left(I+\tfrac{\tau}{2} A_h\right)^{-n}\left(I-\tfrac{\tau}{2} A_h\right)^n v_h,  \quad n=1,2,...,N.
\end{equation}
It can be verified by the energy method that the Crank-Nicolson method is
unconditionally stable, i.e., for any $n\in\mathbb{N}$,
$\|\left(I+\tfrac{\tau}{2} A_h\right)^{-n}\left(I-\tfrac{\tau}{2} A_h\right)^n\| \le 1$.

For the error analysis, we need a result on the rational function
\begin{equation*}
  r_{cn}(z)=\frac{1-\frac{z}{2}}{1+\frac{z}{2}}.
\end{equation*}
\begin{lemma}\label{lem:cnbound}
For any arbitrary $R >0$, there exist $C>0$ and $c>0$ such that
\begin{equation*}
    |e^{-nz} - r_{cn}(z)^n| \le \left\{\begin{array}{ll}
 \displaystyle     Ce^{-\frac{cn}{|z|}},  & \quad |\arg z |\le \delta_1,\, |z| \ge R,\\
     C n |z|^3 e^{-cn|z|},& \quad |\arg z |\le \delta_1,\, |z| \le R,
    \end{array}\right.
\end{equation*}
\end{lemma}
\begin{proof}
The proof of general cases can be found in \cite[Lemmas 9.2 and 9.4]{Thomee:2006}. We
briefly sketch the proof here. By setting $w=1/z$, the first inequality follows from
\begin{equation*}
   r_{cn}(z)=\frac{1-\frac{z}{2}}{1+\frac{z}{2}} = - \frac{1-2w}{1+2w}=-r(4w)=-e^{-4w+O(w^{2})}, \quad w \rightarrow 0,
\end{equation*}
and that for $c \le \cos \delta_1$,
$$ 
  \displaystyle  |e^{-z}| = e^{-\Re z} \le e^{-c|z|} \le Ce^{-\frac{c}{|z|}}, \, |\arg z| \le \delta_1, \ |z|\ge R.
$$ 
The first estimate now follows by the triangle inequality. Meanwhile, we observe that
\begin{equation*}
  \begin{aligned}
    & |r_{cn}(z)-e^{-z}|\le C |z|^{3}, \quad |z|\le R,\, |\arg z| \le \delta_1,\\
    & |r_{cn}(z)|\leq e^{-c|z|},\quad |\arg z|\le \delta_1,\, |z|\le R.
  \end{aligned}
\end{equation*}
Consequently for $z$ under consideration
\begin{equation*}
 |e^{-nz} - r_{cn}(z)^n| = |(e^{-z} - r_{cn}(z))\sum_{j=0}^{n-1}r_{cn}(z)^j e^{-(n-1-j)z}| \le C |z|^{3}ne^{-cn|z|}.
\end{equation*}
This completes the proof of the lemma.
\end{proof}

Now we can state an $L^2\II$-norm estimate for \eqref{eqn:Crank-Nicolson}
in case of smooth initial data.
\begin{theorem}\label{thm:fullsmooth:cn}
Let $u$ and $U^n$ be solutions of problem \eqref{variational} and \eqref{eqn:Crank-Nicolson} with $v\in D(A)$
and $U^0=R_h v$, respectively. Then for $t_n = n\tau$ and any $\beta\in[0,1/2)$, there holds
\begin{equation*}
\| u(t_n)- U^n \|_{L^2\II} \le C (h^{\al-2+2\beta} + \tau^2 t_n^{-1})\| Av \|_{L^2\II}.
\end{equation*}
\end{theorem}
\begin{proof}
Like before, we split the error $e^n$ into
\begin{equation*}
   e^n=  (u(t_n) - u_h(t_n)) + (u_h(t_n) - U^n) := \rlh^n + \vtht^n,
\end{equation*}
where $u_h$ denotes the semidiscrete Galerkin solution with $v_h=R_h v$. Then
by Theorem \ref{thm:semismooth}, the term $\rlh^n$ satisfies the following estimate
\begin{equation*}
   \|\rlh^n \|_{L^2\II} \le C h^{\al-2+2\beta}\| Av\|_{L^2\II}.
\end{equation*}
It remains to bound $\vtht^n = E_h(n\tau)v_h - r_{cn}(\tau A_h)^nv_h$ by
\begin{equation*}
   \|\vtht^n\|_{L^2\II} \le C\tau^2t_n^{-1} \| A_hv_h \|_{L^2\II}.
\end{equation*}
Note that $\tau A_h$ is also sectorial with the same constant
as $A_h$, and further
$$
\| (zI-\tau A_h)^{-1} \| = \tau^{-1} \| \frac{z}{\tau}-A_h \| \le C \frac{1}{|z|}.
$$
With $t_n=n\tau$, it suffices to show
\begin{equation*}
  \| A_h^{-1} (E_h(n) - r_{cn}(A_h)^n)\| \le Cn^{-1}.
\end{equation*}
By Lemma \ref{lem:semigroup}, there holds
\begin{equation*}
   A_h^{-1} r_{cn}(A_h)^n = \frac{1}{2\pi\mathrm{i}} \int_{{\Gamma_{\ep}}
     \cup{\Gamma_{\ep}^R} \cup{\Gamma^R}}r_{cn}(z)^n z^{-1} R(z;A_h) \,dz.
\end{equation*}
Since $ \| r_{cn}(z)^n z^{-1} R(z;A_h)\| =O(z^{-2})$ for large $z$, we can let $R$ tend to $\infty$.
Further, by \cite[Lemma 9.3]{Thomee:2006}, we have
\begin{equation*}
    A_h^{-1} E_h(n) =\frac{1}{2\pi\mathrm{i}} \int_{{\Gamma_{\ep}}\cup{\Gamma_{\ep}^{\infty}}} e^{-nz} z^{-1} R(z;A_h) \,dz.
\end{equation*}
By Lemma \ref{lem:cnbound},
\begin{equation*}
\|(e^{-nz} - r_{cn}(z)^n) z^{-1} R(z;A_h)\| =O(z)\quad \mbox{ as } z\to 0,\ |\arg z| \leq \delta_1,
\end{equation*}
and consequently, by taking $\ep\rightarrow 0$, there holds
\begin{equation*}
\begin{split}
   A_h^{-1} (E_h(n) - r_{cn}(A_h)^n)
    &=\frac{1}{2\pi\mathrm{i}} \int_{\Gamma} (e^{-nz} - r_{cn}(z)^n) z^{-1} R(z;A_h) \,dz,
\end{split}
\end{equation*}
where the sector $\Gamma$ is given by $\Gamma= \left\{z:z=\rho e^{\pm \mathrm{i}\delta_1}, \rho\ge 0 \right\}$.
By applying Lemma \ref{lem:cnbound} with $R=1$, we deduce
\begin{equation} \label{eqn:fullestimate1}
\begin{aligned}
   \|A_h^{-1} (E_h(n) - r_{cn}(A_h)^n) \|
   &= \bigg |\hspace{-.6mm}\bigg |\frac{1}{2\pi\mathrm{i}}
            \int_{\Gamma} (e^{-nz} - r_{cn}(z)^n) z^{-1} R(z;A_h) \,dz \bigg |\hspace{-.6mm}\bigg |\\
   &\le C \int_0^1 \rho n e^{-cn\rho}\, d\rho + C \int_1^{\infty} \rho^{-2}e^{-cn\rho^{-1}}\, d\rho\\
   &\le Cn^{-1} \left(\int_0^{\infty} \rh e^{-\rh} d\rh + \int_0^{\infty} e^{-\rh}\,d\rh\right)\le Cn^{-1}.
\end{aligned}
\end{equation}
This completes the proof of the theorem.
\end{proof}

Now we turn to the case of nonsmooth initial data, i.e., $v\in L^2\II$. It is known that
in case of the standard parabolic equation, the
Crank-Nicolson method fails to give an optimal error estimate for such data unconditionally
because of a lack of smoothing property \cite{LuskinRannacher:1982,Zlamal:1974}. Hence we
employ a damped Crank-Nicolson scheme, which is realized by replacing the first two time steps
by the backward Euler method. Further, we denote
\begin{equation}\label{eqn:dampedCN}
 r_{dcn}(z)^n =r_{bw}(z)^2 r_{cn}(z)^{n-2}.
\end{equation}
The damped Crank-Nicolson scheme is also unconditionally stable. Further, the function
$r_{dcn}(z)$ has the following estimates \cite[Lemma 2.2]{Hansbo:1999}.
\begin{lemma}\label{lem:dcnbound}
Let $r_{dcn}$ be defined as in \eqref{eqn:dampedCN} then there exist
positive constants $\ep$, $R$, $C$, $c$ such that
\begin{equation}\label{eqn:dcnbound}
\begin{split}
    &|r_{dcn}(z)^n| \le \left\{\begin{array}{ll}
       (1+C|z|)^n, & |z|<\ep;\\
       e^{-cn|z|}, & \forall\, |z|\le 1,\, |\arg(z)| \le \delta_1;\\
       C|z|^{-2}e^{-\frac{c(n-2)}{|z|}}, & \forall |z|\ge 1,\, |\arg(z)| \le \delta_1,\, n\ge2;\\
       C|z|^{-2}, & |z| \ge R,\,  n\ge2,\\
       \end{array}\right. \\
    &|r_{bw}(z)^2-e^{-2z}| \le C|z|^2,  \quad \forall |z|\le \ep \quad \text{or} \quad |\arg(z)| \le \delta_1.
\end{split}
\end{equation}
\end{lemma}

\begin{theorem}\label{thm:fullnonsmooth:cn}
Let $u$ be the solution of problem \eqref{variational}, and $U^n= r_{dcn}(\tau A_h)U^0$ with $v\in L^2\II$
and $U^0=P_h v$. Then for $t_n = n\tau$ and any $\beta\in[0,1/2)$, there holds
\begin{equation*}
   \| u(t_n)- U^n \|_{L^2\II} \le C (h^{\al-2+2\beta} t_n^{-1}+ \tau^2 t_n^{-2})\| v \|_{L^2\II}.
\end{equation*}
\end{theorem}
\begin{proof}
We split the error $e^n=u(t_n)-U^n$ as \eqref{eqn:fullsplit}. Since the bound on $\rlh^n$
follows from Theorem  \ref{thm:seminonsmooth}, it remains to bound
$\vtht^n = E_h(\tau n)v_h - r_{dcn}(\tau A_h)^nv_h$ for $n\ge1$ as
\begin{equation*}
   \|  \vtht^n \|_{L^2\II} \le C \tau^2 t_n^{-2} \| v_h \|_{L^2\II}.
\end{equation*}
Let $F_n(z)= e^{-nz}-r_{dcn}(z)^n$. Then it
suffices to show for $n\ge1$
\begin{equation*}
  \|F_n(A_h)\|\le C  n^{-2}.
\end{equation*}
The estimate is trivial for $n=1,2$ by boundedness. For $n > 2$, we split $F_n(z)$ into
\begin{equation*}
\begin{split}
  F_n(z) & = r_{bw}(z)^2(e^{-(n-2)z}-r_{cn}(z)^{n-2})+e^{-(n-2)z}(e^{-2z}-r_{bw}(z)^2) \\
   &:= f_1(z)+f_2(z).
\end{split}
\end{equation*}
It follows from \cite[Lemma 9.1 and Lemma 9.3]{Thomee:2006} that
\begin{equation*}
  \begin{aligned}
    r_{dcn}(A_h)^n &= \frac{1}{2\pi\mathrm{i}} \int_{{\Gamma_{\ep}}\cup{\Gamma_{\ep}^R} \cup{\Gamma^R}}r_{dcn}(z)^nR(z;A_h)\,dz,\\
    E_h(n) &= \frac{1}{2\pi\mathrm{i}} \int_{{\Gamma_{\ep}}\cup{\Gamma_{\ep}^{\infty}}}e^{-nz}R(z;A_h)\,dz.
  \end{aligned}
\end{equation*}
Using the fact $\| r_{dcn}(z)^n R(z;A_h)\| =O(z^{-3})$ as $z\rightarrow \infty$,
we may let $R\to\infty$ to obtain
\begin{equation*}
  F_n(A_h)= \frac{1}{2\pi\mathrm{i}} \int_{{\Gamma_{\ep}}\cup{\Gamma_{\ep}^{\infty}}}F_n(z)R(z;A_h)\,dz.
\end{equation*}
Further, by Lemma \ref{lem:dcnbound}, $\|F_n(z) R(z;A_h)\| =O(z)$ as $z \rightarrow 0$,
and consequently by taking $\ep\rightarrow 0$ and setting $\Gamma=
\left\{z:z=\rho e^{\pm i\delta_1}, \rho\ge 0 \right\}$, we have
\begin{equation}\label{eqn:repret3}
\begin{split}
  F_n(A_h) &= \frac{1}{2\pi\mathrm{i}} \int_{\Gamma}F_n(z)R(z;A_h)\,dz\\
   &= \frac{1}{2\pi\mathrm{i}} \int_{\Gamma}(f_1(z)+f_2(z))R(z;A_h)\,dz.
\end{split}
\end{equation}
Now we estimate the two terms separately. First, by Lemmas \ref{lem:cnbound} and \ref{lem:dcnbound}, we get
\begin{equation*}
\begin{split}
  &|f_1(z)| \le  |r_{dcn}(z)^n| + |r_{bw}(z)^2||e^{-(n-2)z}| \le C|z|^{-2}e^{-\frac{cn}{|z|}}, \quad z\in \Gamma,\ |z|\ge1,\\
  &|f_1(z)| \le |r_{bw}(z)^2||r_{cn}(z)^{n-2}-e^{-(n-2)z}| \le C|z|^{3}ne^{-cn|z|}, \quad z\in \Gamma,\ |z|\le 1.
\end{split}
\end{equation*}
Repeating the argument for \eqref{eqn:fullestimate1} gives that for $n> 2$
\begin{equation*}
  \bigg|\hspace{-0.6mm}\bigg|\frac1{2\pi\mathrm{i}}\int_{\Gamma}f_1(z)R(z;A_h) \,dz \bigg|\hspace{-0.6mm}\bigg| \le Cn^{-2}.
\end{equation*}
As to other term, we deduce from \eqref{eqn:dcnbound} that
\begin{equation*}
   |f_2(z)| \le |e^{-(n-2)z}||r_{bw}(z)^2-e^{-2z}| \le C |z|^2 ,\quad \forall |z| \le \ep,
\end{equation*}
and thus we can change the integration path $\Gamma$ to $\Gamma_{\ep/n}^{\infty}\cup \Gamma_{\ep/n}$.
Further, we deduce from Lemma \ref{lem:dcnbound} that
\begin{equation*}
   |f_2(z)|=|e^{-(n-2)z}(r_{bw}(z)^2-e^{-2z})| \le Ce^{-c(n-2)|z|}|z|^2,\quad \forall z\in\Ga_{\ep/n}^{\infty}.
\end{equation*}
Thus, we derive the following bound for $n>2$
\begin{equation*} 
  \bigg|\hspace{-0.6mm}\bigg|\frac1{2\pi\mathrm{i}}\int_{\Gamma}f_1(z)R(z;A_h) \,dz \bigg|\hspace{-0.6mm}\bigg|
           \le C \int_{\ep/n}^{\infty}e^{-c(n-2)\rho} \rho d\rho +C\int_{\Ga_{\ep/n}}\rho \, d\rho \le Cn^{-2}.
\end{equation*}
This completes the proof of the theorem.
\end{proof}

\section{Numerical results}\label{sec:numeric}
In this section, we present numerical experiments to verify our theoretical results. To this end,
we consider the following three examples:
\begin{enumerate}
 \item[(a)] smooth initial data: $v(x)= x(x-1)$, which lies
  in $\Hd{3/2-\ep}$.
 \item[(b)] nonsmooth initial data: (b1) $ v(x)=\chi_{(1/2,1)}(x)$,
 the characteristic function of the interval $(1/2,1)$;
 (b2) $v(x)=x^{1/4}$; Note that in (b1) $v \in \Hd{1/2-\epsilon}$
   while in (b2) $v \in \Hd {1/4-\epsilon}$, for any $\epsilon >0$.
 \item[(c)] discontinuous potential $q(x)=\chi_{(0,1/2)}(x)$. 
\end{enumerate}

We examine separately the spatial and temporal convergence rates at $t=1$.
For the case of nonsmooth initial data, we are especially interested in the errors
for $t$ close to zero, and thus we also present the errors at $t=0.1$, $0.01$,
$0.005$, and $0.001$. The exact solutions to these examples are not available
in closed form, and hence we compute the reference solution on a very refined mesh.
We measure the accuracy of the numerical approximation $U^n$ by the normalized errors
$\|u(t_n)-U^n\|_{L^2\II}/\|v\|_{L^2\II}$ and $\|u(t_n)-U^n\|_{\Hd {\al/2}}/ \|v
\|_{L^2\II}$. The normalization enables us to observe the behavior of the errors with respect
to time in case of nonsmooth initial data. To study the rate of convergence in space,
we use a time step size $\tau=10^{-5}$ so that the time discretization error is negligible,
and we have the space discretization error only.

\subsection{Numerical results for example (a): smooth initial data}

In Table \ref{tab:smoothBE} we show the errors $\|u(t_n)-U^n\|_{L^2\II}$ and $\|u(t_n)-U^n
\|_{\Hd {\al/2}}$ with the backward Euler method. We have set $\tau=10^{-5}$, so that the
error incurred by temporal discretization is negligible. In the table, \texttt{ratio}
refers to the ratio of the errors when the mesh size $h$ (or time step size $\tau$)
halves, and the numbers in the bracket denote theoretical convergence rates.
The numerical results show $O(h^{\al-1/2})$ and $O(h^{\al/2-1/2})$ convergence rates for
the $L^2\II$- and $\Hd {\al/2}$-norms of the error, respectively. In Fig.
\ref{fig:smooth_space}, we plot the results for $\al=1.5$ at $t=1$
in a log-log scale. The $\Hd{\alpha/2}$-norm estimate is fully confirmed, but the $L^2
\II$-norm estimate is suboptimal: the empirical convergence rate is one half order higher
than the theoretical one. The suboptimality is attributed to the low regularity of the
adjoint solution, used in Nitsche's trick. In view of the singularity of the term
$x^{\alpha-1}$ in the solution representation, cf. Remark \ref{rmk::singular}, the
spatial discretization error is concentrated around the origin.

\begin{table}[hbt!]
\small
\caption{$L^2$- and $ \tilde{H}^{\al/2}$-norms of the error for example (a),
smooth initial data, with $\al=1.25, 1.5, 1.75$ for backward Euler method and $\tau=10^{-5}$; in the last
column in brackets is the theoretical rate.} 
\label{tab:smoothBE}
\begin{center}
\vspace{-.3cm}
     \begin{tabular}{|c|c|cccccc|c|}
     \hline
   $\al$  & $h$ & $1/16$ & $1/32$ &$1/64$ &$1/128$ & $1/256$ & 1/512&ratio \\
     \hline
     $1.25$ & $L^2$  &5.13e-3 &2.89e-3 &1.69e-3 &1.00e-3 &6.03e-4 &3.71e-4 &$\approx$ 0.76 ($0.25$) \\
     \cline{2-8}
     & $\tilde{H}^{\al/2}$         &4.93e-2 &4.39e-2 &3.98e-2 &3.62e-2 &3.29e-2 &3.00e-2 &$\approx$ 0.14 ($0.13$)\\
    \hline
     $1.5$ & $L^2$   &3.62e-4 &1.70e-4 &8.37e-5 &4.17e-5 &2.09e-5 &1.06e-5 &$\approx$ 1.02 ($0.50$)\\
     \cline{2-8}
     & $\tilde{H}^{\al/2}$     &7.57e-3 &6.25e-3 &5.20e-3 &4.33e-3 &3.58e-3 &2.91e-3 &$\approx$ 0.27 ($0.25$)\\
     \hline
     $1.75$ & $L^2$   &1.12e-5 &4.61e-6 &1.92e-6 &8.02e-7 &3.35e-7 &1.37e-7 &$\approx$ 1.26 ($0.75$)\\
     \cline{2-8}
     & $\tilde{H}^{\al/2}$   &4.63e-4 &3.47e-4 &2.58e-4 &1.95e-4 &1.46e-4 &1.06e-4 &$\approx$ 0.42 ($0.38$)\\
     \hline
     \end{tabular}
\end{center}
\end{table}

In Table \ref{tab:smoothChecktime}, we let the spacial step size $h \rightarrow 0$
and examine the temporal convergence order, and observe an $O(\tau)$ and $O(\tau^2)$ convergence
rates for the backward Euler method and Crank-Nicolson method, respectively.
Note that for the case $\al=1.75$, the Crank-Nicolson method fails to achieve
an optimal convergence order. This is attributed to the fact that $v$ is not in the domain
of the differential operator ${_0^RD_x^{\al}}$
for $\al > 1.5$. In contrast, the damped Crank-Nicolson method yields the
desired $O(\tau^2)$ convergence rates, cf. Table \ref{tab:DampedCNsmooth}.
This confirms the discussions in Section \ref{ssec:fullcn}.

\begin{figure}[htb!]
\center
  \includegraphics[trim = 1cm .1cm 2cm 0.5cm, clip=true,width=10cm]{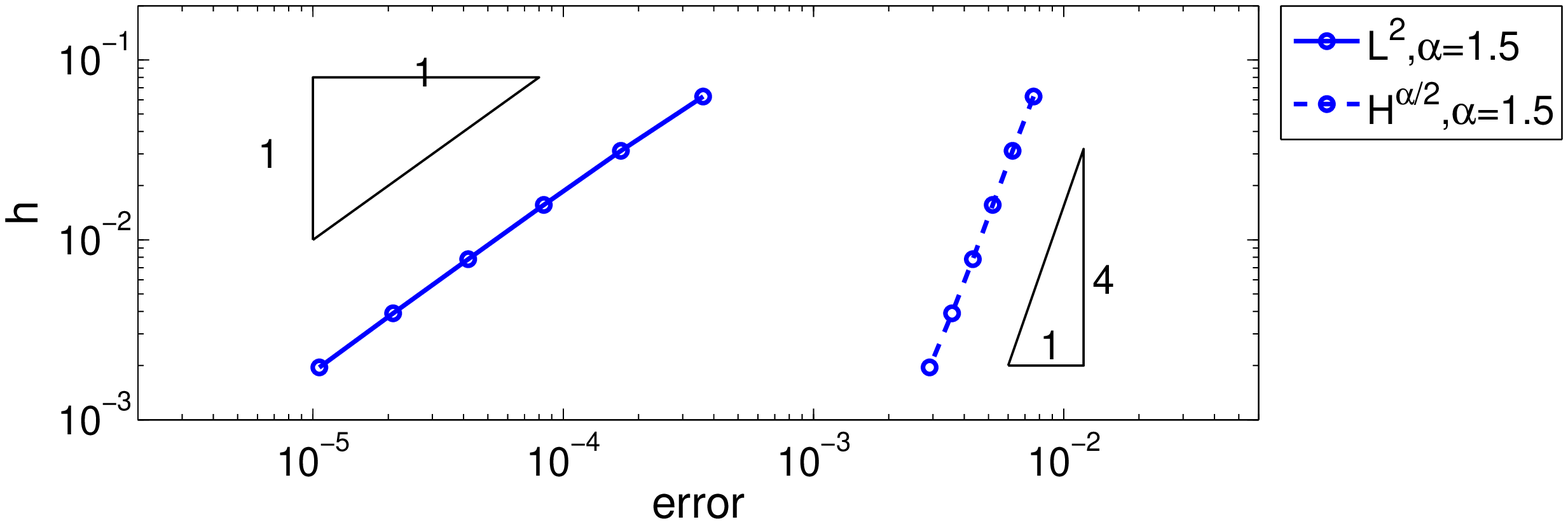}
  \vspace{-.2cm}
  \caption{Numerical results for example (a) (smooth data) with $\alpha=1.5$ at $t=1$.}\label{fig:smooth_space}
\end{figure}

\begin{table}[htb!]
\caption{$L^2$-norm of the error for example (a),
non-smooth initial data, with $\al=1.25, 1.5, 1.75$,
$h=2\times 10^{-5}$ (BE - backward Euler, CN - Crank-Nicolson)}
\label{tab:smoothChecktime}
\begin{center}
\vspace{-.3cm}
     \begin{tabular}{|c|c|ccccc|c|}
     \hline
     & $\tau$ & $1/10$ & $1/20$ &$1/40$ &$1/80$ & $1/160$  &ratio \\
     \hline
     BE
     & $\al=1.25$      &3.01e-2 &1.41e-2 &6.63e-3 &3.10e-3 &1.41e-3  &$\approx$ 1.10 (1.00) \\
     \cline{2-7}
     & $\al=1.5$       &1.32e-2 &5.88e-3 &2.71e-3 &1.25e-3 &5.62e-4  &$\approx$ 1.13 (1.00)  \\
     \cline{2-7}
     & $\al=1.75$      &4.79e-3 &1.88e-3 &7.95e-3 &3.53e-4 &1.55e-4  &$\approx$ 1.20 (1.00)  \\
     \hline
     CN
     & $\al=1.25$      &3.18e-3 &5.98e-4 &1.35e-4 &3.32e-5 &8.52e-6  &$\approx 2.10$ (2.00)\\
     \cline{2-7}
     & $\al=1.5$       &3.22e-3 &7.32e-4 &1.75e-4 &4.32e-5 &1.05e-5  &$\approx$ 2.06 (2.00)  \\
     \cline{2-7}
     & $\al=1.75$      &3.67e-3 &1.09e-3 &3.33e-4 &1.08e-4 &3.09e-5  &$\approx$ 1.73 ( - - ) \\
     \hline
     \end{tabular}
\end{center}
\end{table}

\begin{table}[htb!]
\caption{$L^2$-norm of the error for example (a),
non-smooth initial data, for damped Crank-Nicolson method
with $\al=1.75$  
and $h=2\times 10^{-5}$.}
\label{tab:DampedCNsmooth}
\begin{center}
\vspace{-.3cm}
     \begin{tabular}{|c|ccccc|c|}
     \hline
      $\tau$ & $1/10$ & $1/20$ &$1/40$ &$1/80$ & $1/160$  &ratio \\
     \hline
      $\al=1.75$  &7.57e-4 &1.98e-4 &5.45e-5 &1.40e-5 &2.90e-6  &$\approx$ 1.98 (2.00) \\
     \hline
     \end{tabular}
\end{center}
\end{table}

\subsection{Numerical results for nonsmooth initial data: example (b)}
In Tables \ref{tab:nonsmooth1BE}, \ref{tab:nonsmooth1Checktime} and \ref{tab:nonsmooth1smalltime},
we present numerical results for problem (b1). Table \ref{tab:nonsmooth1BE} shows that the spatial
convergence rate is of the order $O(h^{\al-1+\beta})$ in $L^2\II$-norm and $O(h^{\alpha/2-1+\beta})$
in $\Hd{\alpha/2}$, whereas Table \ref{tab:nonsmooth1Checktime} shows that the temporal convergence
order is of order $O(\tau)$ and $O(\tau^2)$ for the backward Euler method and damped Crank-Nicolson
method, respectively. For the case of nonsmooth initial data, we are interested in the errors for
$t$ closed to zero, thus we check the error at $t=0.1$, $0.01$, $0.005$ and $0.001$. From Table
\ref{tab:nonsmooth1smalltime}, we observe that both the $L^2\II$-norm and $\Hd{\al/2}$-norm of
the error exhibit superconvergence, which theoretically remains to be established. Numerically,
for this example. one observes that the solution is smoother than in $\Hdi 0{\alpha-1+\beta}$ for
small time $t$, cf. Fig. \ref{fig:checksolution_b1}.

Similarly, the numerical results for problem (b2) are presented in Tables \ref{tab:nonsmooth2BE},
\ref{tab:nonsmooth2Checktime} and \ref{tab:nonsmooth2smalltime}; see also Fig. \ref{fig:nonsmooth2_small_time}
for a plot of the results in Table \ref{tab:nonsmooth2smalltime}. It is observed that the convergence
is slower than that for problem (b1), due to the lower solution regularity.

\begin{table}[htb!]
\small
\caption{$L^2$- and $ \tilde{H}^{\al/2}$-norms of the error for example (b1), nonsmooth initail data,
for backward Euler method with  $\tau=10^{-5}$.}
\label{tab:nonsmooth1BE}
\begin{center}
\vspace{-.3cm}
     \begin{tabular}{|c|c|cccccc|c|}
     \hline
     $\al$& $h$ & $1/16$ & $1/32$ &$1/64$ &$1/128$ & $1/256$ & 1/512&ratio \\
     \hline
     $1.25$ & $L^2$ &6.65e-3 &3.75e-3 &2.18e-3 &1.29e-3 &7.78e-4 &4.77e-4 &$\approx$ 0.76 ($0.25$)\\
     \cline{2-8}
     &$\tilde{H}^{\al/2}$         &6.36e-2 &5.66e-2 &5.12e-2 &4.66e-2 &4.24e-2 &3.87e-2 &$\approx$ 0.14 ($0.13$)\\
    \hline
     $1.5$ & $L^2$   &3.78e-4 &1.77e-4 &8.56e-5 &4.22e-5 &2.09e-5 &1.04e-5 &$\approx$ 1.03 ($0.50$)\\
     \cline{2-8}
     & $\tilde{H}^{\al/2}$         &7.31e-3 &6.01e-3 &5.00e-3 &4.16e-3 &3.43e-3 &2.79e-3 &$\approx$ 0.27 ($0.25$) \\
     \hline
     $1.75$ & $L^2$  &2.11e-5 &9.49e-6 &4.06e-6 &1.69e-6 &6.83e-7 &2.59e-7 &$\approx$ 1.27 ($0.75$)\\
     \cline{2-8}
     & $\tilde{H}^{\al/2}$         &3.63e-4 &2.69e-4 &1.99e-4 &1.50e-4 &1.12e-4 &8.19e-5 &$\approx$ 0.43 ($0.38$) \\
     \hline
     \end{tabular}
\end{center}
\end{table}

\begin{table}[htb!]
\caption{$L^2$-norm of the error for example (b1),
non-smooth initial data, with 
$h=2\times 10^{-5}$ (BE - backward Euler, CN - Crank-Nicolson)
}
\label{tab:nonsmooth1Checktime}
\begin{center}
\vspace{-.3cm}
     \begin{tabular}{|c|c|ccccc|c|}
     \hline
     & $\tau$ & $1/10$ & $1/20$ &$1/40$ &$1/80$ & $1/160$  &ratio \\
     \hline
     BE
     & $\al=1.25$      &3.73e-2 &1.80e-2 &8.53e-3 &4.00e-3 &1.81e-3  &$\approx$ 1.09 (1.00)\\
     \cline{2-7}
     & $\al=1.5$       &1.26e-2 &5.64e-3 &2.59e-3 &1.20e-3 &5.40e-4  &$\approx$ 1.13 (1.00)\\
     \cline{2-7}
     & $\al=1.75$      &3.68e-3 &1.44e-3 &6.12e-3 &2.71e-4 &1.20e-4  &$\approx$ 1.19 (1.00)\\
     \hline
     CN
     & $\al=1.25$      &3.52e-3 &9.10e-4 &2.39e-4 &5.90e-5 &1.30e-5  &$\approx 2.01$ (2.00)\\
     \cline{2-7}
     & $\al=1.5$       &8.86e-4 &2.42e-2 &6.46e-5 &1.61e-5 &3.44e-6  &$\approx$ 1.99 (2.00)\\
     \cline{2-7}
     & $\al=1.75$      &1.86e-4 &4.01e-5 &1.02e-5 &2.57e-6 &5.41e-7  &$\approx$  2.09 (2.00)\\
     \hline
     \end{tabular}
\end{center}
\end{table}

\begin{table}[htb!]
\small
\caption{$L^2$- and $ \tilde{H}^{\al/2}$-norms of the error for example (b1), nonsmooth initial data,
with $\al=1.5$ 
for backward Euler method and $\tau=10^{-5}$.} 
\label{tab:nonsmooth1smalltime}
\begin{center}
\vspace{-.3cm}
     \begin{tabular}{|c|c|cccccc|c|}
     \hline
     $t$ & $h$ & $1/16$ & $1/32$ &$1/64$ &$1/128$ & $1/256$ & 1/512&ratio \\
     \hline
     $0.1$ & $L^2$     &3.64e-3 &1.53e-3 &7.42e-4 &3.72e-5 &1.87e-4 &9.46e-5 &$\approx$ 1.04 ($0.50$) \\
     \cline{2-8}
     & $\tilde{H}^{\al/2}$        &7.00e-2 &5.59e-2 &4.62e-2 &3.87e-2 &3.21e-2 &2.61e-2 &$\approx$ 0.28 ($ 0.25$) \\
    \hline
     $0.01$ & $L^2$   &2.81e-2 &7.07e-2 &1.63e-3 &3.84e-4 &9.21e-5 &2.18e-5 &$\approx$ 2.07 ($ 0.50$)\\
     \cline{2-8}
     & $\tilde{H}^{\al/2}$        &4.04e-1 &1.56e-1 &6.09e-2 &2.49e-2 &1.03e-2 &4.27e-3 &$\approx$ 1.31 ($ 0.25$) \\
     \hline
     $0.005$ & $L^2$  &4.27e-2 &1.45e-2 &3.44e-3 &7.95e-4 &1.88e-4 &4.41e-4 &$\approx$ 2.07 ($0.50$) \\
     \cline{2-8}
     & $\tilde{H}^{\al/2}$         &5.94e-1 &3.34e-1 &1.27e-1 &5.05e-2 &2.08e-2 &8.56e-3 &$\approx$ 1.26 ($0.25$) \\
     \hline
     $0.001$ & $L^2$  &1.41e-1 &5.22e-2 &1.64e-2 &4.47e-3 &1.02e-3 &2.32e-3 &$\approx$ 1.80 ($0.50$) \\
     \cline{2-8}
     & $\tilde{H}^{\al/2}$         &2.61e0 &1.45e0 &6.63e-1 &2.81e-1 &1.08e-1 &4.34e-2 &$\approx$ 1.20 ($0.25$) \\
     \hline
     \end{tabular}
\end{center}
\end{table}

\begin{figure}[htb!]
\center
  \includegraphics[clip=true,width=13cm]{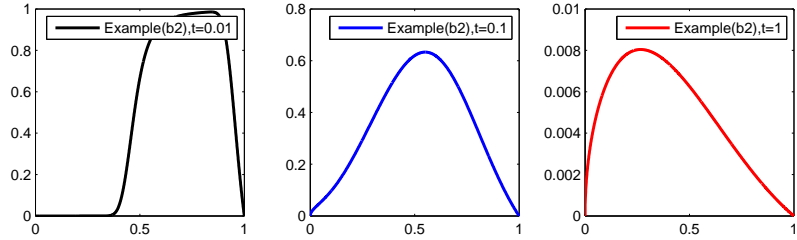}
  \vspace{-.6cm}
  \caption{Solution profile of example (b1) with $\alpha=1.5$ at
   $0.01$, $0.1$ and $1$.}
   \label{fig:checksolution_b1}
\end{figure}

\begin{table}[htb!]
\small
\caption{$L^2$- and $ \tilde{H}^{\al/2}$-norms of the error for example (b2), nonsmooth initial data,
for backward Euler method with $\tau=10^{-5}$.}
\label{tab:nonsmooth2BE}
\begin{center}
\vspace{-.3cm}
     \begin{tabular}{|c|c|cccccc|c|}
     \hline
     $\al$& $h$ & $1/16$ & $1/32$ &$1/64$ &$1/128$ & $1/256$ & 1/512&ratio \\
     \hline
     $1.25$ & $L^2$  &6.31e-3 &3.55e-3 &2.07e-3 &1.23e-3 &7.38e-4 &4.53e-4 &$\approx$ 0.76 ($0.25$) \\
     \cline{2-8}
     & $\tilde{H}^{\al/2}$        &6.03e-2 &5.37e-2 &4.86e-2 &4.42e-2 &4.02e-2 &3.67e-2 &$\approx$ 0.14 ($0.13$) \\
    \hline
     $1.5$ & $L^2$   &4.11e-4 &1.91e-4 &9.24e-5 &4.55e-5 &2.26e-5 &1.12e-5 &$\approx$ 1.03 ($0.50$) \\
     \cline{2-8}
     & $\tilde{H}^{\al/2}$         &7.88e-3 &6.48e-3 &5.39e-3 &4.48e-3 &3.70e-3 &3.01e-3 &$\approx$ 0.27 ($0.25$) \\
     \hline
     $1.75$ & $L^2$  &2.75e-5 &1.21e-6 &5.09e-6 &2.11e-6 &8.48e-7 &3.20e-7 &$\approx$ 1.28 ($0.75$) \\
     \cline{2-8}
     & $\tilde{H}^{\al/2}$        &4.50e-4 &3.33e-4 &2.46e-4 &1.86e-4 &1.39e-4 &1.01e-4 &$\approx$ 0.42 ($0.38$) \\
     \hline
     \end{tabular}
\end{center}
\end{table}

\begin{table}[htb!]
\caption{$L^2$-norm of the error for example (b2),
non-smooth initial data, with 
$h=2\times 10^{-5}$ (BE - backward Euler, CN - Crank-Nicolson).}
\label{tab:nonsmooth2Checktime}
\begin{center}
\vspace{-.3cm}
     \begin{tabular}{|c|c|ccccc|c|}
     \hline
     & $\tau$ & $1/10$ & $1/20$ &$1/40$ &$1/80$ & $1/160$  &ratio \\
     \hline
     BE
     & $\al=1.25$      &3.57e-2 &1.71e-2 &8.09e-3 &3.80e-3 &1.71e-3  &$\approx$ 1.09 (1.00) \\
     \cline{2-7}
     & $\al=1.5$       &1.36e-2 &6.82e-3 &2.80e-3 &1.30e-3 &5.81e-4  &$\approx$ 1.13 (1.00) \\
     \cline{2-7}
     & $\al=1.75$      &4.55e-3 &1.78e-3 &7.57e-3 &3.35e-4 &1.48e-4  &$\approx$ 1.20 (1.00) \\
     \hline
     CN
     & $\al=1.25$      &3.32e-3 &8.59e-4 &2.26e-4 &5.60e-5 &1.24e-5  &$\approx$ 2.03 (2.00)\\
     \cline{2-7}
     & $\al=1.5$       &9.36e-4 &2.59e-5 &6.95e-5 &1.74e-6 &3.80e-7  &$\approx$ 1.99 (2.00) \\
     \cline{2-7}
     & $\al=1.75$      &1.69e-4 &4.43e-5 &1.22e-5 &3.15e-6 &6.50e-7  &$\approx$ 1.99 (2.00) \\
     \hline
     \end{tabular}
\end{center}
\end{table}

\begin{table}[htb!]
\small
\caption{$L^2$- and $ \tilde{H}^{\al/2}$-norms of the error for example (b2), nonsmooth initial data,
for backward Euler method with $\tau=10^{-5}$.}
\label{tab:nonsmooth2smalltime}
\begin{center}
\vspace{-.3cm}
     \begin{tabular}{|c|c|cccccc|c|}
     \hline
     $t$& $h$ & $1/16$ & $1/32$ &$1/64$ &$1/128$ & $1/256$ & 1/512&ratio \\
     \hline
     $0.1$ & $L^2$  &1.73e-2 &8.56e-3 &4.27e-3 &2.14e-3 &1.08e-3 &5.43e-4 &$\approx$ 1.00 (0.50) \\
     \cline{2-8}
     & $\tilde{H}^{\al/2}$         &3.83e-1 &3.20e-1 &2.67e-1 &2.23e-1 &1.84e-1 &1.50e-1 &$\approx$ 0.26 (0.25) \\
    \hline
     $0.01$ & $L^2$   &3.35e-2 &1.39e-2 &6.45e-3 &3.17e-3 &1.58e-3 &7.97e-4 &$\approx$ 1.07 (0.50)\\
     \cline{2-8}
     & $\tilde{H}^{\al/2}$         &6.41e-1 &4.89e-1 &3.97e-1 &3.28e-1 &2.71e-1 &2.20e-1 &$\approx$ 0.30 (0.25) \\
     \hline
     $0.005$ & $L^2$ &4.23e-2 &1.83e-2 &7.65e-3 &3.61e-3 &1.79e-3 &8.96e-4 &$\approx$ 1.11 (0.50) \\
     \cline{2-8}
     & $\tilde{H}^{\al/2}$         &7.52e-1 &5.89e-1 &4.55e-1 &3.71e-1 &3.04e-1 &2.47e-1 &$\approx$ 0.29 (0.25) \\
     \hline
     $0.001$ & $L^2$  &1.07e-1 &4.12e-2 &1.54e-2 &5.89e-3 &2.49e-3 &1.19e-3 &$\approx$ 1.30 (0.50) \\
     \cline{2-8}
     & $\tilde{H}^{\al/2}$        &1.98e0 &1.19e0 &7.51e-1 &5.19e-1 &4.08e-1 &3.28e-1 &$\approx$ 0.52 (0.25) \\
     \hline
     \end{tabular}
\end{center}
\end{table}

\begin{figure}[htb!]
\center
  \includegraphics[trim = 1cm .1cm 2cm 0.5cm, clip=true,width=10cm]{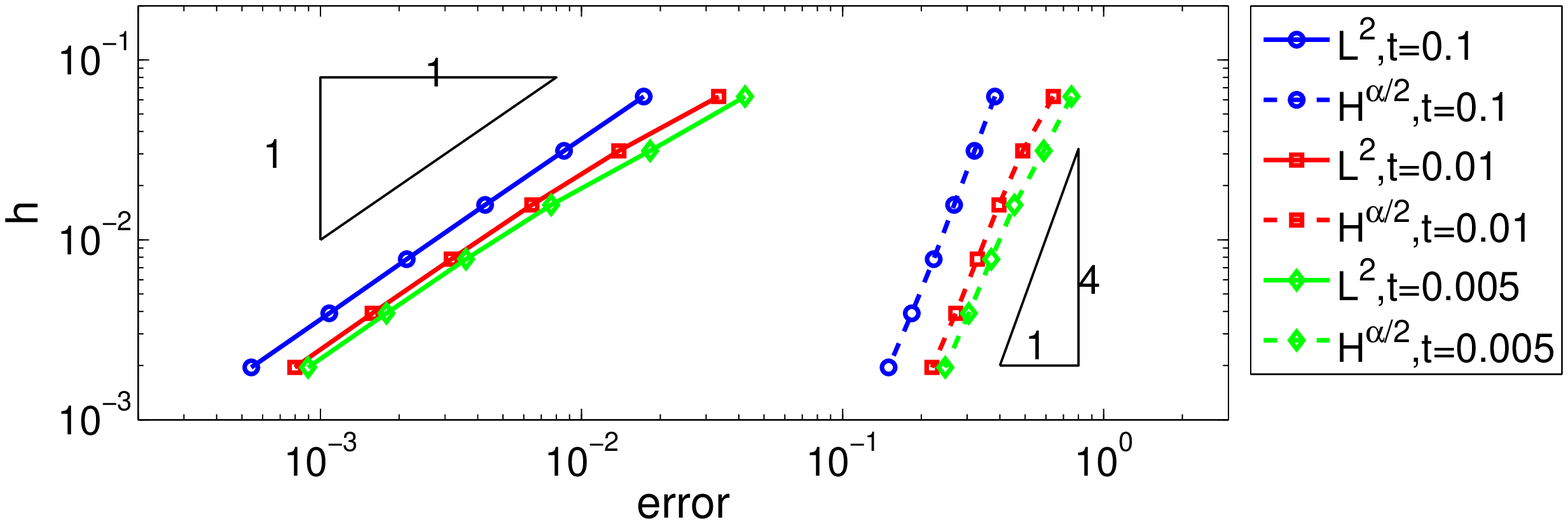}
  \vspace{-.2cm}
  \caption{Numerical results for example (b2)
  with $\alpha=1.5$ at $t=0.1, 0.01$ and $0.005$.}\label{fig:nonsmooth2_small_time}
\end{figure}

\subsection{Numerical results for general problems: example (c)}
Our theory can easily extend to problems with a potential function $q \in L^{\infty}\II$.
Garding's inequality holds for the bilinear form, and thus all theoretical results follow
by the same argument. The normalized $L^2\II$- and $\Hd{\al/2}$-norms of the spatial error
are reported in Table \ref{tab:general} at $t=1$ for $\al=1.25$, $1.5$ and $1.75$.
The results concur with the preceding convergence rates.

\begin{table}[htb!]
\small
\caption{$L^2$-norm of the error for the general differential equation
non-smooth initial data, example (c), with 
$\tau=2\times 10^{-5}$ (BE - backward Euler, CN - Crank-Nicolson).}
\label{tab:general}
\begin{center}
\vspace{-.3cm}
     \begin{tabular}{|c|c|cccccc|c|}
     \hline
     $\al$& $h$ & $1/16$ & $1/32$ &$1/64$ &$1/128$ & $1/256$ & 1/512&ratio \\
     \hline
     $1.25$ & $L^2$  &4.80e-3 &2.71e-3 &1.58e-3 &9.40e-4 &5.66e-4 &3.48e-4 &$\approx$ 0.76 (0.25) \\
     \cline{2-8}
     & $\tilde{H}^{\al/2}$         &4.62e-2 &4.12e-2 &3.73e-2 &3.39e-2 &3.09e-2 &2.82e-2 &$\approx$ 0.14 (0.13) \\
    \hline
     $1.5$ & $L^2$   &2.75e-4 &1.31e-4 &6.50e-5 &3.24e-5 &1.63e-5 &8.20e-5 &$\approx$ 1.00 (0.50) \\
     \cline{2-8}
     &$\tilde{H}^{\al/2}$  &5.90e-3 &6.86e-3 &4.05e-3 &3.37e-3 &2.79e-3 &2.26e-3 &$\approx$ 0.27 (0.25) \\
     \hline
     $1.75$ & $L^2$ &7.88e-6 &3.19e-6 &1.33e-6 &5.58e-7 &2.34e-7 &9.60e-8 &$\approx$ 1.27 (0.75) \\
     \cline{2-8}
     & $\tilde{H}^{\al/2}$        &3.24e-4 &2.42e-4 &1.80e-4 &1.36e-4 &1.02e-4 &7.43e-5 &$\approx$ 0.42 (0.38) \\
     \hline
     \end{tabular}
\end{center}
\end{table}

\section{conclusion}\label{sec:conclusion}
In this paper, we have studied a finite element method for an initial boundary value problem for
the parabolic problem with a space fractional derivative of Riemann-Liouville type and order $\alpha\in
(1,2)$ using the analytic semigroup theory. The existence and uniqueness of a weak solution in $L^2
(0,T;\Hd{\alpha/2})$ were established, and an improved regularity result was also shown. Error estimates
in the $L^2\II$- and $\Hd {\alpha/2}$-norm were established for a space semidiscrete scheme with a piecewise
linear finite element method, and $L^2\II$-norm estimates for fully discrete schemes based on the backward
Euler method and the Crank-Nicolson method, for both smooth and nonsmooth initial data.

The numerical experiments fully confirmed the convergence of the numerical schemes, but the $L^2\II$-norm
error estimates are suboptimal: the empirical convergence rates are one-half order higher
than the theoretical ones. This suboptimality is attributed to the inefficiency of Nitsche's trick,
as a consequence of the low regularity of the adjoint solution. Numerically, we observe that the
$\Hd{\alpha/2}$-norm convergence rates agree well with the theoretical ones. The optimal
convergence rates in the $L^2\II$-norm and the $\Hd{\alpha/2}$-norm estimate for the fully
discrete schemes still await theoretical justifications.

\section*{Acknowledgements}
The research of B. Jin has been supported by NSF Grant DMS-1319052, R. Lazarov was supported in parts
by NSF Grant DMS-1016525 and J. Pasciak
has been supported by NSF Grant DMS-1216551. The work of all authors was also supported
by Award No. KUS-C1-016-04, made by King Abdullah University of Science and Technology (KAUST).

\appendix

\section{Proof of Theorem \ref{thm:existence}}\label{app:existence}

\begin{proof}
We divide the proof into four steps.

\noindent Step (i) (energy estimates for $u_m$).
Upon taking $u_m$ as the test function, the identity $2(u_m',u_m)=\frac{d}{dt}\|u_m\|_{L^2\II}^2$ for a.e. $0 \le t \le T$, and the
coercivity of $A(\cdot,\cdot)$, we deduce
\begin{equation}\label{appenergy1}
    \frac{d}{dt}\|u_m(t)\|_{L^2\II}^2 + c_0\|u_m(t)\|^2_{\Hd{\al/2}} \le 2\|f(t)\|_{H^{-\al/2}\II}\|u_m(t)\|_{\Hd{\al/2}}.
\end{equation}
Young's inequality and integration in $t$ over $(0,t)$ gives
\begin{equation*}
  \max_{0\le t\le T} \|u_m(t)\|_{L^2\II}^2 \le \| v \|_{L^2\II}^2 +
C \| f \|_{L^2(0,T;H^{-\al/2}\II)}^2 .
\end{equation*}
Next we integrate \eqref{appenergy1} from $0$ to $T$, and repeat the argument to get
\begin{equation}\label{eqn:uineq}
  \| u_m \|_{L^2(0,T;\Hd {\al/2})}^2 
  \le  \| v \|_{L^2\II}^2  + C\| f \|_{L^2(0,T;H^{-\al/2}\II)}^2 .
\end{equation}
Finally we bound $\| u_m' \|_{L^2(0,T;H^{-\al/2}\II)}$. For any $\fy \in \Hd{\al/2}$
such that  $\|\fy\|_{\Hd{\al/2}} \le 1$, we decompose it into
$\fy=P\fy + (I-P)\fy$ with $P\fy \in \text{span}\{\omega_k\}_{k=1}^m$ and
$I-P \in \text{span}\{\omega_k \}_{k>m}$.
By the stability of the projection $P$, $\|P\fy\|_{\Hd{\al/2}}
\le C\|\fy\|_{\Hd{\al/2}} \le C$, it follows from
$(u_m',P\fy)+A(u_m,P\fy)=(f,P\fy)$ and $(u_m', P \fy)=(u_m',\fy)$ that
\begin{equation*}
  \begin{aligned}
    |\langle u_m' (t), \fy \rangle|  = |\langle u_m' (t) ,P\fy\rangle|
              &  \le  C\left(\|f(t) \|_{H^{-\al/2}\II}+\|u_m(t)\|_{\Hd {\al/2}}\right).
  \end{aligned}
\end{equation*}
Consequently, by the duality argument and \eqref{eqn:uineq} we arrive at
\begin{equation}\label{eqn:utineq}
     \| u_m' \|_{L^2(0,T;H^{-\al/2}\II)}^2 \le C \left(\| f \|_{L^2(0,T;H^{-\al/2}\II)}^2 + \| v \|_{L^2\II}^2\right ).
\end{equation}

\noindent Step (ii) (convergent subsequence).
By \eqref{eqn:uineq} and \eqref{eqn:utineq}, there exists a subsequence, also
denoted by $\{u_m\}$, and
$u \in L^2(0,T;\Hd {\al/2})$ and $\tilde{u} \in L^2(0,T;H^{-\al/2}\II)$, such that
\begin{equation}\label{weakconvg}
    \begin{split}
      u_{m} &\rightarrow u \quad \text{weakly in } L^2(0,T;\Hd{\al/2}), \\
      u_{m}' &\rightarrow \tilde{u} \quad \text{weakly in }  L^2(0,T;H^{-\al/2}\II).
    \end{split}
\end{equation}
By choosing $\phi \in C_0^{\infty}[0,T]$ and $\psi \in \Hd {\al/2}$, we deduce
$$
\int_0^T\langle u_m',\phi\psi \rangle \,dt=-\int_0^T\langle u_m,\phi'\psi \rangle \,dt.
$$
By taking $m\rightarrow \infty$ we obtain
\begin{equation*}
   \int_0^T\langle \tilde{u},\phi\psi \rangle \,dt=-\int_0^T\langle u,\phi'\psi \rangle \,dt
    =\int_0^T\langle u',\phi\psi \rangle \,dt.
\end{equation*}
Thus $\tilde u=u'$ by the density of $\{ \phi(t)\psi(x) \}$ in $ L^2(0,T;\Hd{\al/2})$.

\noindent Step (iii) (weak form). Now for a fixed integer $N$, we choose a test
function $\psi \in V_N = \text{span}\{\omega_k\}_{k=1}^N$, and $\phi\in C^\infty[0,T]$.
Then for $m\ge N$, there holds
\begin{equation}\label{weak4}
   \int_0^T \langle u_m',\psi\phi\rangle +A(u_m,\psi)\phi \, dt = \int_0^T \langle f,\psi\phi \rangle \, dt.
\end{equation}
Then letting $m \rightarrow \infty$, \eqref{weakconvg}
and the density of $\{ \phi(t)\psi(x) \}$ in $ L^2(0,T;\Hd{\al/2})$ gives
\begin{equation}\label{weak2}
\int_0^T \langle u',\fy\rangle +A(u,\fy) \, dt = \int_0^T \langle f,\fy \rangle \, dt,\quad \forall  \fy \in L^2(0,T;\Hd {\al/2}).
\end{equation}
Consequently, we arrive at
\begin{equation*}
   \langle u',\fy\rangle +A(u,\fy) = \langle f,\fy \rangle,\ \ \forall \fy \in \Hd {\al/2}\quad \,\mbox{a.e.  } 0\leq t \leq T.
\end{equation*}

\noindent(iv) (initial condition). The argument presented in \cite[Theorem 3, pp.
287]{Evans:2010} yields  $u \in C([0,T];L^2\II)$. By taking $\phi \in C^{\infty}[0,T]$
with $\fy(T)=0$ and $\psi\in \text{span}\{\omega_k\}_{k=1}^N$, integrating \eqref{weak4} and \eqref{weak2}
by parts with respect to $t$, and a standard density argument, we arrive at the initial condition $u(0)=v$.
The uniqueness follows directly from the energy estimates.
\end{proof}
\bibliographystyle{abbrv}
\bibliography{frac}
\end{document}